%% file: Reconstruction.tex
\documentclass[a4paper,11pt]{article}
\title{A Non-dissipative Reconstruction Scheme for the Compressible Euler Equations}
\author{Nina Aguillon}

\input{preambuleenglish}

\begin{document}
\maketitle

\begin{abstract}
 We present a finite volume scheme, first on the Burgers equations, then on the Euler equations, based on a conservative reconstruction of shocks inside each cells of the mesh. Its main features are the following:
 \begin{itemize}
 \item the scheme is exact whenever the initial datum is a pure shock, in the sense that the approximate solution is the exact solution averaged over the cells of the mesh;
 \item in general, the scheme has a very low numerical diffusion and the shocks have a width of one or two cells;
 \item no spurious oscillations in the momentum appear behind slowly moving shocks, which is not the case in most of the scheme developed so far;
\end{itemize}
We also present prospective result on the full Euler equations with energy. The wall heating phenomenon, which is an artificial elevation of the temperature when a shock reflects on a wall, is also drastically diminished.
\end{abstract}

\bigskip

\noindent  \textbf{Key phrases:} Shock capturing scheme; Euler equations; slowly moving shocks; wall heating phenomenon


\section*{Introduction}

When approaching the solution of hyperbolic system, each finite volume scheme introduces a small amount of dissipation. The main effect of this numerical viscosity is to smear out the discontinuities over several cells. But in some cases, even the simplest finite volume scheme, as Lax-Friedrichs's or Godunov's, fails to approach the solution in a satisfactory way. In~\cite{Q94}, Quirk gives a list of such failings. We recall two of them in the context of gas dynamics. The first one is the apparition of an artificial spike in the momentum, followed by spurious oscillations, when computing a slowly moving shock (i.e. a shock that travels slowly compared to the speed of sound). The second one is the wall heating phenomenon, which is a nonphysical spike on the internal energy when a shock reflects on a solid wall, or when two shocks interact. It is worth noticing that for the slowly moving shocks as for the shock reflections, higher order schemes can lead to less satisfactory results than the first order schemes.  In both those cases, the numerical viscosity is held responsible (see~\cite{JL96, AR97} and~\cite{Q94, N87}). Several fixes have been proposed, based on the idea of adding a sufficient amount of dissipation to the scheme, either ``by hand'' or by interlacing two different schemes, see for example~\cite{DM96}.

In this paper, we propose a finite volume scheme that is exact when the initial datum is a pure shock (or contact discontinuity). As a consequence, it approaches exactly slowly moving shocks, and no momentum spikes appear when computing a Riemann problem containing a slowly moving shock. The wall heating phenomenon seems to be eliminated as well. Another nice feature of our scheme is that it reconstructs perfectly isentropic compression. The numerical simulations show that our scheme is of order $2$ on Riemann problems containing only shocks, which is, up to our knowledge, a novelty. 

The main idea of our scheme is the following. Suppose that the initial datum corresponds to a pure shock located at an interface of the mesh. After one iteration in time, at time $\Delta t$, this shock has moved inside a cell. As finite volume scheme are designed to be conservative, the numerical solution will be average in that cell, creating an intermediate value that does not correspond to a pointwise value of the exact solution. Roughly speaking, our scheme consists in rebuilding the initial shock inside that cell. Then, the reconstructed shocks are advected during a time $\Delta t$ and the corresponding fluxes are computed. This type of scheme has been first introduced in~\cite{DL02} on scalar conservation law with an increasing convex flux , and then used to compute non classical shocks in~\cite{BCLL08}. A general framework for framework and a stability analysis can be found in~\cite{L12} and~\cite{L08}. Other types of discontinuous reconstruction has been applied to systems, and particularly to gas dynamics, in~\cite{AVCL08} and~\cite{LWM08}.

Our aim is to generalized this idea to a wider class of hyperbolic equations. In comparison with~\cite{L08}, we would first like to suppress the assumption of monotonicity on the flux, and allow waves to travel in every direction. We must carefully avoid waves interaction near an interface. This can be done in a very simple way with a moving mesh, where the interfaces are enough bend, so that every wave travel slower than the mesh. The Rusanov scheme and its higher order extensions can be interpreted this way. Second, we attempt to apply this method to system of conservation laws. The solution of the Riemann problem consists in the succession of several waves  (shocks, contact discontinuities or rarefactions) and it is a priori unclear how to reconstruct the solution. In every cell, we choose to reconstruct the solution as a single shock or contact discontinuity.

The paper is organized as follows. In the first section, we present the reconstruction scheme on a single conservation law with convex flux, typically the Burgers equation. We describe the three steps of the scheme: reconstruction, advection and march in time, and insist on the necessity to use a moving mesh to avoid wave interaction. In the second section, we extend this scheme to the barotropic Euler equations. The key point is to choose the wave that has to be reconstructed. We present a simple criterion that we believe is relevant and robust. The reconstruction scheme degenerates to Lax-Friedrichs scheme in smooth regions. We explain how to coupled it with higher order schemes to regain accuracy in those regions, and present various numerical tests. Remarkably, the spurious oscillations behind a slowly moving shock are completely eliminated. 
The third and last section is devoted to prospective results  on the Euler equations with energy. We present some simulations of shock reflections, where the wall heating phenomenon is eliminated. This results are encouraging and show that the reconstruction scheme do extend to the full Euler equation. However this simulations are done with a rather abrupt detection of the contact discontinuities, and could probably be improved with a smarter criterion to decide which wave should be reconstructed. In a forthcoming version of this paper, we will discuss the two dimensional extension on the shallow water equation.

\paragraph{Acknowledgments}
The author heartily thanks Frédéric Lagoutière for his enlightening advice and guidance during the achievement of this work.

\section{Presentation of the scheme on scalar conservation law}

We consider the Cauchy problem for a scalar conservation law
\begin{equation} \label{ScalarCP}
 \left\{
 \begin{array}{l}
 \partial_{t} u + \partial_{x} f(u)= 0 \  \text{ for } t \in \R_{+} \text{ and } x \in \R \\\
 u_{| t=0}=u^{0}. 
\end{array}
\right.
\end{equation}
where the flux $f$ is $\mathcal{C}^{1}(\R)$ and the initial data $u^{0}$ are in $L^{\infty}(\R)$. We are interested in entropy solution, i.e. weak solutions of~\eqref{ScalarCP} that verify the additional entropy inequality
\begin{equation} \label{EntropyIneq}
  \partial_{t} S(u) + \partial_{x} G(u) \leq 0 \  \text{ for } t \in \R_{+} \text{ and } x \in \R.
\end{equation}
The functions $(S,G)$ are an entropy-entropy flux pair: both $S$ and $G$ are $\mathcal{C}^{1}$, $S$ is convex and $G'=S' f'$. If $f$ and $S$ are strictly convex, then this solution is unique, as proven for example in~\cite{K70, GR91}.

Let us now consider a finite volume scheme for~\eqref{ScalarCP}. We denote by $t^{0}=0<t^{1}< \cdots < t^{n} < \cdots$ the successive times where the solution is computed, and by $\Delta t^{k}= t^{k+1}-t^{k}$ the corresponding time steps. We subdivide the space $\R$ in cells having all the same size $\Delta x$. We denote by $x_{j}= j \Delta x$ the centers of the cells and by $x_{j+1/2}= x_{j} + \frac{\Delta x}{2}$ their extremities. Finally, we denote by $u_{j}^{n}$ an approximation of the exact solution at time $t^{n}$, averaged on cell $j$:
$$u_{j}^{n} \approx \frac{1}{\Delta x} \int_{x_{j-1/2}}^{x_{j+1/2}} u(t^{n},x) dx.$$ 
Integrating~\eqref{ScalarCP} on the slab $[t^{n}, t^{n+1}] \times [x_{j-1/2}, x_{j+1/2}]$,  we obtain the family of finite volume schemes
\begin{equation} \label{FVS}
\left\{
\begin{array}{l}
  u_{j}^{n+1}= u_{j}^{n}- \frac{\Delta t^{n}}{\Delta x} \left( f_{j+1/2}^{n} - f_{j-1/2}^{n} \right), \ \ n \geq 0, \\[10pt]
  u_{j}^{0} = \frac{1}{\Delta x} \int_{x_{j-1/2}}^{x_{j+1/2}} u^{0}(x) dx,
\end{array}
\right.
\end{equation}
where the numerical flux $f_{j+1/2}^{n}$ is an approximation of $\int_{t^{n}}^{t^{n+1}} f(u(t, x_{j+1/2} )) dt$ that characterized the scheme. It remains to explain how the fluxes are chosen in our scheme. There is not much difference with the case of linear advection presented, in term of discontinuous reconstruction scheme, in~\cite{BCLL08}. 
The idea is depicted on Figure~\ref{FNotationsBurgers}. We consider that the value $u_{j}^{n}$ in cell $j$ is actually coming from the averaging of a shock between $u_{j-1}^{n}$ and $u_{j+1}^{n}$ located somewhere inside the cell. Starting from $u_{j}^{n}$ we can reconstruct this discontinuity. By conservation of mass, it should lie at a distance
$$d_{j}^{n}= \Delta x \frac{u_{j+1}^{n}-u_{j}^{n}}{u_{j+1}^{n}-u_{j-1}^{n}}$$
of the left interface. It falls inside the cell if and only if $u_{j}^{n}$ takes value between its neighbors $u_{j+1}^{n}$ and $u_{j-1}^{n}$. We do not want to reconstruct a shock that is not entropy satisfying. Therefore, as $f$ is convex, we accept the reconstruction only if $u_{j-1}^{n} \geq u_{j+1}^{n}$. 
\begin{psfrags} 
 \psfrag{u_j}{$u_{j}$}
 \psfrag{u_j+1}{$u_{j+1}$}
 \psfrag{u_j-1}{$u_{j-1}$}
 \psfrag{u_j,R}{$u_{j,R}$}
 \psfrag{u_j,L}{$u_{j,L}$}
 \psfrag{j}{$j$}
 \psfrag{j+1}{$j+1$}
 \psfrag{j-1}{$j-1$}
 \psfrag{d_j}{$d_{j}$}
 \psfrag{dx-d_j}{$\Delta x - d_{j}$}
 \begin{figure}[H]
 \centering
 \includegraphics[width=10cm]{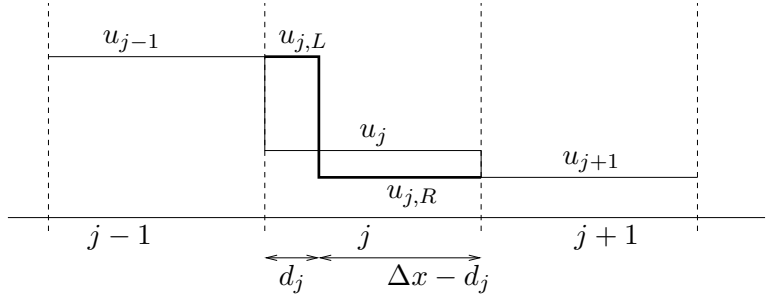}
 \caption{Conservative reconstruction of a shock.}  \label{FNotationsBurgers}
 \end{figure}
\end{psfrags}
We now let the reconstructed discontinuities evolve during a time $\Delta t$ and compute the fluxes at each interface. Suppose for a minute that $\min_{x \in \R} f'(u^{0}(x))>0$. Then every reconstructed shock travels to the right and, if two such shocks interact, the resulting waves will also travel to the right. As a consequence, if we perform a reconstruction in cell $j$, we know that the reconstructed shock will cross the right interface at time
$$ \Delta t^{j+1/2} = \frac{\Delta x - d_{j}^{n}}{ \sigma_{j}^{n}} $$
where $\sigma_{j}^{n}= \frac{f(u_{j-1}^{n})-f(u_{j+1}^{n})}{u_{j-1}^{n}-u_{j+1}^{n}}$ is the speed of the reconstructed shock. As everything is coming from the left, the flux passing through the $j+1/2$ interface is $f(u_{j+1}^{n})$ before the crossing time and $f(u_{j-1}^{n})$ afterwards (if the crossing time is smaller than the time step $\Delta t$). 

Suppose now that it exists a sonic point, i.e. that $f'(u^{0})$ changes sign. If $u^{0}$ is decreasing, we can reconstruct in every cell with entropy satisfying shocks. Then by convexity of $f$, $\sigma_{j}^{n}$ is bigger than $\sigma_{j+1}^{n}$ and it is possible that the two waves interact near their common interface in $x_{j+1/2}$, as depicted on Figure~\ref{FInteractionWaves}. Moreover, the resulting wave can travel in any direction, and it becomes impossible to compute the flux through this interface without solving the new Riemann problem created by the interaction of the waves. 
\begin{psfrags} 
 \psfrag{j}{$j$}
 \psfrag{j+1}{$j+1$}
  \psfrag{s_j}{$\sigma_{j}^{n}$}
 \psfrag{s_j+1}{$\sigma_{j+1}^{n}$}
 \begin{figure}[H] 
 \centering
 \includegraphics[width=10cm]{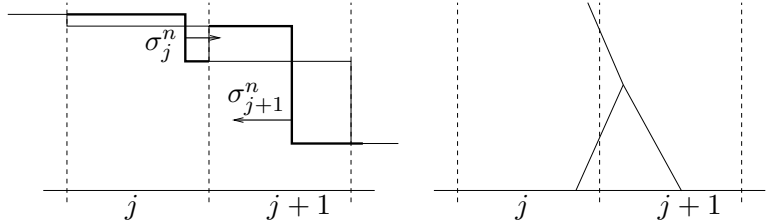}
 \caption{ Interaction of waves near an interface. On the left, the bold lines represent the reconstruction and the arrows the speed of the shocks. On the left, the auto similar waves created by the reconstructed shocks are drew in the $(x,t)$ plane.} \label{FInteractionWaves}
 \end{figure}
 \end{psfrags}
To overcome this difficulty we use a moving mesh. The property of propagation at finite speed states that the speed of waves are smaller than $V_{\text{waves}} = \max_{j \in \Z}| f'(u_{j})|$. We introduce a mesh speed $V_{\text{mesh}} \geq V_{\text{waves}}$ and bend the interface according to it. It follows that the shock reconstructed in cell $\mathcal{C}_{j}$  will cross the left interface. Moreover, even though two adjacent waves interact, the resulting waves will also travel slower than $V_{\text{mesh}}$. Therefore, they will not catch up with the interface, as depicted on Figure~\ref{FMovingMesh}. At the next time iteration, the mesh is bent in the other direction so that the space grid is unchanged after two time steps.
\begin{psfrags} 
 \begin{figure}[H] 
 \centering
 \includegraphics[width=10cm]{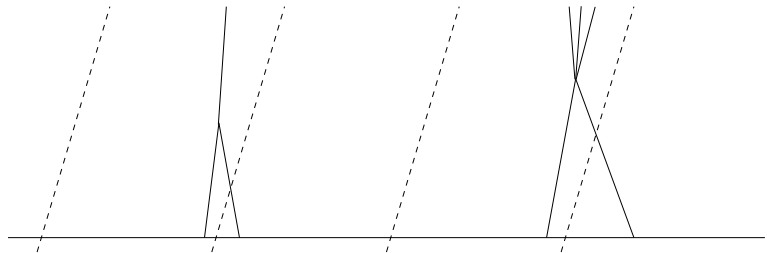}
 \caption{ A mesh with interfaces traveling faster than the maximum of wave speeds. Two interactions between adjacent waves occur, but the resulting perturbations of the solution do not cross the interfaces.} \label{FMovingMesh}
 \end{figure}
 \end{psfrags}
 This use of a moving mesh makes possible the computation of the flux without resolving any wave interaction. The very same idea is the key stone of the central schemes. Indeed, if $|V_{mesh}|>V_{waves}$, the evaluation of the solution of Riemann problems along the $x= V_{mesh}t$ becomes a trivial task. This is used in the Rusanov scheme, where $V_{mesh}$ is locally computed, and in its higher order extension initiated in~\cite{NT90} .
 
 \subsection{Reconstruction scheme}
 Summarizing the previous section, we split the reconstruction scheme into $3$ steps.
 
 \emph{Step $1$: Reconstruction}
 
\noindent We first try to reconstruct, for every $j\in \Z$, the value $u_{j}^{n}$ as the average of an entropy satisfying shock between $u_{j-1}^{n}$ and $u_{j+1}^{n}$. If it is not possible, nothing is done. In other words, we compute the distance to the left interface of the cell by conservation of mass
\begin{equation} \label{Distance}
d_{j}^{n}= \Delta x \frac{u_{j+1}^{n}-u_{j}^{n}}{u_{j+1}^{n}-u_{j-1}^{n}},
\end{equation}
and we accept the reconstruction if and only if it corresponds to an entropy satisfying shock lying inside the cell. We introduce the reconstructed states 
\begin{equation} \label{ReconstructedStates}
 \left\{
 \begin{array}{lcll}
 u_{j,L}^{n}= u_{j-1}^{n} &\text{ and } &u_{j,R}^{n}= u_{j+1}^{n}	& \text{ if } 0 < d_{j}^{n} < \Delta x \text{ and } u_{j-1}^{n}> u_{j+1}^{n}, \\
 u_{j,L}^{n}= u_{j}^{n} &\text{ and } &u_{j,R}^{n}= u_{j}^{n}, 		&\text{ otherwise } ,
\end{array}
\right. 
\end{equation}
and the speed associated to the reconstructed shock
\begin{equation} \label{BSpeed}
\sigma_{j}^{n}=
 \left\{
 \begin{array}{ll}
 \frac{f(u_{j,L})-f(u_{j,R})}{u_{j,L}-u_{j,R}}	& \text{ if } 0 < d_{j}^{n} < \Delta x \text{ and } u_{j-1}^{n}> u_{j+1}^{n}, \\
f'(u_{j})											&\text{ otherwise. } 
\end{array}
\right. 
\end{equation}

 \emph{Step $2$: Computation of the fluxes}
 
\noindent We suppose at time $t^{n}$,  the mesh speed $V_{\text{mesh}}^{n}$ is given such that 
\begin{equation} \label{meshspeed}
|V_{\text{mesh}}^{n}| \geq V_{\text{waves}} ^{n}= \max_{j \in \Z}| f'(u_{j}^{n})|.
\end{equation}
In practice, $V_{\text{mesh}}^{n}$ will change of sign at each iteration. The finite volume scheme is now obtained by integrating~\eqref{ScalarCP} on the slab 
$$\{ (t,x):  \  t^{n} \leq t \leq t^{n+1}, \  x_{j-1/2}^{n}+ t V_{\text{mesh}}^{n}  \leq x \leq  x_{j+1/2}^{n}+ t V_{\text{mesh}}^{n}   \} ,$$
where the extremities of the space cells are now defined from time to time by $x_{j+1/2}^{n+1}=x_{j+1/2}^{n}+ V_{\text{mesh}}^{n} \Delta t$. Therefore  $f_{j+1/2}^{n}$ is now an approximation of
$$ \int_{t^{n}}^{t^{n+1}} f(u(t, x_{j+1/2}^{n}+ V_{\text{mesh}}^{n} t))- V_{\text{mesh}}^{n} u(t, x_{j+1/2}^{n}+ V_{\text{mesh}}^{n} t) dt.$$
Our construction of the numerical fluxes is as follow. If $V_{\textrm{mesh}}^{n}<0$, the reconstructed shocks cross the right interface $x= x_{j+1/2}^{n}+V_{\text{mesh}}^{n} t$. The crossing time $\Delta t _{j+1/2}^{n+1/2}$ verifies
$$ x_{j-1/2}^{n}+ d_{j}+ \sigma_{j}^{n} \Delta t _{j+1/2}^{n+1/2}= x_{j+1/2}^{n}+V_{\textrm{mesh}}^{n} \Delta t _{j+1/2}^{n+1/2}, $$
and is therefore equals to
$$ \Delta t_{j+1/2}^{n+1/2}= \frac{\Delta x - d_{j}^{n}}{\sigma_{j}^{n}-V_{\text{mesh}}^{n}}. $$
Consequently, we set
\begin{equation} \label{flux1/2}
\begin{aligned}
 \Delta t f_{j+1/2}^{n}=  & \min(\Delta t, \Delta t_{j+1/2}^{n+1/2})  \left( f(u_{j,R}^{n}) - V_{\textrm{mesh}}^{n} u_{j,R}^{n} \right) \\
 & +  (\Delta t - \min(\Delta t, \Delta t_{j+1/2}^{n+1/2})) \left( f(u_{j,L}^{n}) - V_{\textrm{mesh}}^{n} u_{j,L}^{n} \right).
\end{aligned}
\end{equation}
If $V_{\textrm{mesh}}^{n}>0$, then the reconstructed shocks cross the left interface and we denote by
$$ \Delta t_{j-1/2}^{n+1/2}= \frac{d_{j}}{V_{\text{mesh}}^{n}-\sigma_{j}^{n}} $$
the crossing time. The flux is given by
\begin{equation} \label{flux-1/2}
\begin{aligned}
\Delta t f_{j-1/2}^{n}=  & \min(\Delta t, \Delta t _{j-1/2}^{n+1/2}) \left( f(u_{j,L}^{n})- V_{\textrm{mesh}}^{n} u_{j,L}^{n} \right) \\
 & +  (\Delta t - \min(\Delta t, \Delta t _{j-1/2}^{n+1/2})) \left( f(u_{j,R}^{n}) - V_{\textrm{mesh}}^{n} u_{j,R}^{n} \right).
\end{aligned}
\end{equation}

\emph{Step $3$: March in time}

\noindent We update compute $u$ at the next time by applying
 $$ u_{j}^{n+1}= u_{j}^{n}- \frac{\Delta t^{n}}{\Delta x} \left( f_{j+1/2}^{n} - f_{j-1/2}^{n} \right), \ \ n \geq 0 .$$

\begin{rem}
 When no reconstruction is performed, either because $u_{j-1}^{n}$ and $u_{j+1}^{n}$ are linked by a rarefaction wave, or because $d_{j}$ does not belong to $[0, \Delta x ]$, the value of $d_{j}$ does not matter. The flux degenerates toward a simple left or right decentering
 $$ f_{j+1/2}^{n}= 
 \left\{
 \begin{array}{ll}
 f(u_{j}^{n})- V_{\textrm{mesh}}^{n} u_{j}^{n} & \textrm{ if } V_{\textrm{mesh}}^{n}<0 \\
 f(u_{j+1}^{n})- V_{\textrm{mesh}}^{n} u_{j+1}^{n} & \textrm{ if } V_{\textrm{mesh}}^{n}>0
\end{array}
\right.
 $$
 which also corresponds to the fluxes of the staggered Lax-Friedrichs scheme, which is convergent under Hypothesis~\eqref{meshspeed}.
\end{rem}
\begin{prop}
 The reconstruction scheme is consistent: if $u_{j}^{n}=u_{j+1}^{n}$, then $f_{j+1/2}^{n}= f(u_{j}^{n})- V_{\text{mesh}}^{n} u_{j}^{n}$.
\end{prop}
\begin{proof}
 Indeed, if $u_{j}^{n}=u_{j+1}^{n}$, no reconstruction is performed.
\end{proof}
The stability and convergence of reconstruction scheme for scalar conservation law has recently been studied in~\cite{L12}, where two extensions of the scheme inside rarefaction waves are also proposed.
\begin{rem}
The use of a staggered mesh imposes the CFL condition
 $$ \Delta t^{n} \leq \frac{\Delta x}{|V_{\text{mesh}}^{n}|+V_{\text{waves}}}. $$
\end{rem}

\begin{prop} \label{PureShock} When the initial data correspond to a pure shock, i.e. has the form
$$ u_{|t=0}(x) = u_{l} \mathbf{1}_{x<0} + u_{r} \mathbf{1}_{x>0}, $$
 with $u_{l}>u_{r}$, then the reconstruction scheme is exact, in the sense that
 $$ u_{j}^{n}= \frac{1}{\Delta x} \int_{x_{j-1/2}}^{x_{j+1/2}} u_{ \textrm{exa}}(n \Delta t,x) dx. $$
The function 
 $$u_{ \textrm{exa}}(t,x)=u_{l} \mathbf{1}_{x<\frac{f(u_{l})-f(u_{r})}{u_{l}-u_{r}}t} + u_{r} \mathbf{1}_{x>\frac{f(u_{l})-f(u_{r})}{u_{l}-u_{r}}t} $$ 
 is the exact solution.
\end{prop}
\begin{proof} The scheme is exactly build to that purpose. For the sake of simplicity we suppose that $u_{l}>u_{r}>0$. Therefore for every $u \in [ u_{r}, u_{l}]$, every wave generated by a Riemann problem between $u_{l}$ an $u$ or between $u$ and $u_{r}$ has a positive speed and it is not necessary to use a moving mesh. We set $V_{ \textrm{mesh}}^{n}=0$ for all $n$. Let us proceed by induction.  Suppose that for some $n \geq 0$, for every $j \in \Z$,
$$ u_{j}^{n}= \frac{1}{\Delta x} \int_{x_{j-1/2}}^{x_{j+1/2}} u_{ \textrm{exa}}(n \Delta t,x) dx. $$
It is true for $n=0$ as soon as the initial sampling is done correctly. We denote by $j_{0}$ the cell where the shock is located and by $\delta \in [0, \Delta x]$ its distance to  $x_{j_{0}-1/2}$. We have
$$
\begin{cases}
 u_{j}^{n}= u_{l}& \mbox{ if $j<j_{0}$},  \\
 u_{j_{0}}^{n}= \frac{\delta}{\Delta x}u_{l}  + \frac{\Delta x -\delta}{\Delta x} u_{r}& \mbox{ if $j=j_{0}$},  \\
 u_{j}^{n}= u_{r} & \mbox{ if $j>j_{0}$}. \\
\end{cases}
$$
As a consequence, the scheme reconstructs the shock in cell $j_{0}$ with $d_{j_{0}}= \delta$, and it has  the correct speed 
$$\sigma_{ j_{0}}= \frac{f(u_{r})- f(u_{l})}{u_{r}-u_{l}}.$$ 
Moreover, as $u_{j_{0}-2}^{n}=u_{j_{0}-1}^{n}$ and $u_{j_{0}+1}^{n}= u_{j_{0}+2}^{n}$, no reconstruction is performed in cells $j_{0}-1$ and $j_{0}+1$, and
$$f_{j_{0}-1/2}^{n}= f(u_{l})= \frac{u_{l}^{2}}{2},$$
while
$$f_{j_{0}+3/2}^{n}= f(u_{r})= \frac{u_{r}^{2}}{2}.$$
We now compute the flux $f_{j_{0}+1/2}^{n}$. The reconstructed discontinuity crosses the right interface at time
$$ \Delta t_{j_{0}+1/2}^{n+1/2}= \frac{(\Delta x - \delta)}{\sigma_{j_{0}}}. $$
If $\Delta t_{j_{0}+1/2}$ is smaller than $\Delta t$, then Formula~\eqref{flux1/2} rewrites
$$
\begin{aligned}
 \Delta t f_{j_{0}+1/2}^{n} &= \Delta t_{j_{0}+1/2}^{n+1/2} f(u_{r}) + (\Delta t -  \Delta t_{j_{0}+1/2}^{n+1/2}) f(u_{l})  \\
 			&=  \frac{(\Delta x - \delta)}{\sigma_{j_{0}}}[f(u_{r})- f(u_{l}) ] + \Delta t f(u_{l}) \\
			&= \frac{(\Delta x - \delta) (u_{r}-u_{l})}{f(u_{r})-f(u_{l})} [f(u_{r})- f(u_{l}) ] + \Delta t f(u_{l}).
\end{aligned}
$$ 
We obtain
$$ f_{j_{0}+1/2}^{n} = \frac{\Delta x- \delta}{\Delta t} (u_{r}-u_{l}) + f(u_{l}), $$
and therefore we have
$$
\begin{aligned}
u_{j_{0}}^{n+1} &=u_{j_{0}}^{n+1} - \frac{\Delta t}{\Delta x}(f_{j_{0}+1/2}^{n}-f_{j_{0}-1/2}^{n}) \\
 			&= \frac{\delta}{\Delta x}u_{l}  + \frac{\Delta x -\delta}{\Delta x} u_{r}- \frac{\Delta t}{\Delta x} \left( \frac{\Delta x- \delta}{\Delta t} (u_{r}-u_{l}) + f(u_{l}) -f(u_{l}) \right) \\
			&=u_{l} 
\end{aligned}
$$
and, if we denote by $\delta '$ the quantity $\delta + \Delta t \sigma_{j_{0}}- \Delta x$,
$$
\begin{aligned}		
u_{j_{0}+1}^{n+1} 	&=u_{j_{0}+1}^{n+1} - \frac{\Delta t}{\Delta x}(f_{j_{0}+3/2}^{n}-f_{j_{0}+1/2}^{n}) \\
 				&= u_{r}- \frac{\Delta t}{\Delta x} \left( f(u_{r}) - \left[ \frac{\Delta x- \delta}{\Delta t} (u_{r}-u_{l}) + f(u_{l}) \right] \right) \\
				&= \frac{\delta '}{\Delta x} u_{l} +  \frac{\Delta x -\delta '}{\Delta x} u_{r}.
\end{aligned}
$$ 
The crossing time $ \Delta t_{j_{0}+1/2}^{n+1/2}$ is smaller than $\Delta t$ if and only if $ \delta + \sigma_{j_{0}} \Delta t$ is larger than $\Delta x$. Therefore, in that case at time $n+1$, the shock has moved inside the cell $j_{0}+1$ and its distance to the left interface is exactly $\delta '$. On the other hand, when $ \Delta t_{j_{0}+1/2}^{n+1/2} \geq 0$, the shock is still in cell $j_{0}$ at time $n+1$. Its distance to the left interface is $\delta '= \delta  + \Delta t \sigma_{j_{0}}$. In that case, the flux $f_{j_{0}+1/2}^{n}$ is simply $f(u_{r})$, and we easily obtain $u_{j_{0}+1}^{n+1}=u_{r}$ and
$$
\begin{aligned}
u_{j_{0}}^{n+1} &=u_{j_{0}}^{n+1} - \frac{\Delta t}{\Delta x}(f_{j_{0}+1/2}^{n}-f_{j_{0}-1/2}^{n}) \\
 			&= \frac{\delta}{\Delta x}u_{l}  + \frac{\Delta x -\delta}{\Delta x} u_{r}- \frac{\Delta t}{\Delta x}(f(u_{r}) -f(u_{l})) \\
			&= \frac{\delta '}{\Delta x} u_{l} +  \frac{\Delta x - \delta '}{\Delta x} u_{r}, 
\end{aligned}
$$
which concludes the proof.
\end{proof}

\subsection{Numerical test: isentropic compression}
We perform a numerical test that exhibit the non diffusive behavior of the reconstruction scheme. We consider the Burgers equation $f(u)= \frac{u^{2}}{2}$ with the initial data
$$ u^{0}(x)= 
\begin{cases}
 3 & \text{ if } x \leq -3; \\
 3-(x+3) & \text{ if } -3 \leq x \leq -1; \\
 1 & \text{ if } 1 \leq x.
\end{cases}
$$
The exact solution is given by
$$ u(t,x)= 
\begin{cases}
3* \mathbf{1}_{x \leq -3+3t} + \left( 3- \frac{x-(3+3t)}{(1-t)} \right) \mathbf{1}_{-3+3t < x < -1+t } +  \mathbf{1}_{ -1+t \leq x} & \text{ if } t<1 \\
3 * \mathbf{1}_{x \leq 2(t-1)} + \mathbf{1}_{2( t-1)< x} & \text{ if } t\geq1
\end{cases}
$$
\begin{figure}[H]
 \centering \includegraphics[width=14cm]{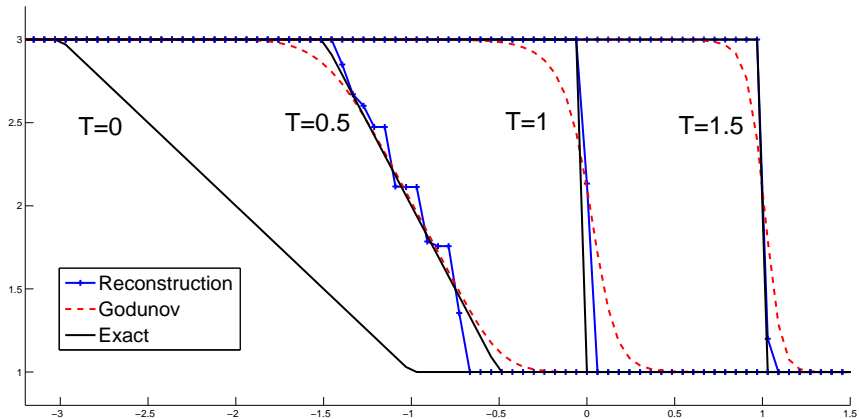}
 \caption{Simulation of an isentropic compression with the Godunov scheme and the reconstruction scheme} \label{FDetenteEnvers}
\end{figure}
We compare the Godunov scheme and the reconstruction scheme on the interval $[-4, 2]$ discretized with $100$ cells. The CFL number is $0.4$. As depicted on Figure~\ref{FDetenteEnvers}, the shock is perfectly reconstructed at time $T=1$ with the reconstruction scheme, while it is smeared out on $10$ cells by the Godunov's scheme. After time $T=1$, the shock is exactly advected by the reconstruction scheme, while the Godunov scheme's is still diffusing it. The apparition of steps inside the isentropic compression is very similar to those which appears in the smooth regions in~\cite{DL02} and~\cite{L12}. It does not affect the order of convergence of the scheme. Entropy fixes have been proposed in~\cite{L12}, and no such steps appear in~\cite{LWM08}

\section{Extension to the isothermal Euler equation}

\subsection{Adaptation of the reconstruction scheme}

This section is devoted to the generalization of the reconstruction scheme to a particular system of conservation laws, the isothermal Euler equations. It describes the evolution of an inviscid compressible gas, having density $\rho$ and velocity $u$, when the pressure law is $p=p(\rho)= c^{2} \rho$. Here, $c$ is the speed of sound. This pressure law to avoids us to deal with vacuum, but the scheme extend straightforwardly to pressure laws $p(\rho)= a \rho^{\gamma}$ with $a>0$ and $\gamma>1$ when there is no vacuum. The Cauchy problem for this system writes 
\begin{equation} \label{EulerEq}
 \begin{cases}
 \partial_{t} \rho + \partial_{x} (\rho u)= 0, \\
 \partial_{t} (\rho u) + \partial_{x} (\rho u^{2} + c^{2} \rho) =0, \\
 \rho_{| t=0}= \rho^{0} \text{ and } u_{| t=0}= u^{0}.
\end{cases}
\end{equation}
This system is strictly hyperbolic, the eigenvalues of the jacobian matrix being $u-c$ and $u+c$. The corresponding fields are genuinely nonlinear. This system has been widely studied. We are here interesting in the structure of the Riemann problem that can be found, for example, in~\cite{Evans}. We recall that a Riemann problem is~\eqref{EulerEq}, with initial data
$$ 
\begin{cases}
 \rho^{0}(x) &= \rho_{L} \mathbf{1}_{x<0}+\rho_{R} \mathbf{1}_{x>0}, \\
 u^{0}(x) &= u_{L} \mathbf{1}_{x<0}+ u_{R} \mathbf{1}_{x>0}.
\end{cases}
 $$
The solution consists of the succession of two waves, rarefactions or shocks, separated by an intermediate state. As a consequence, unless we start with very specific initial data, there is no reason for states $(\rho_{j-1}^{n}, u_{j-1}^{n})$ and $(\rho_{j+1}^{n}, u_{j+1}^{n})$ to be linked by an entropy satisfying shock, and the Riemann problem between these two states is more likely to develop the full pattern of two waves. How to decide for a reconstruction procedure in this framework? We decided to try to reconstruct the state $(\rho_{j}^{n}, u_{j}^{n})$ as the average of \emph{one of the waves} appearing in the Riemann problem between its neighbor states $(\rho_{j-1}^{n}, u_{j-1}^{n})$ and $(\rho_{j+1}^{n}, u_{j+1}^{n})$. This wave must be an entropy satisfying shock and must somehow be the prevailing wave in the Riemann problem. We use the following lemma to decide wether we try to reconstruct and which wave is chosen.
 \begin{lemma} \label{Detect}
 If $u_{L}>u_{R}$ and $\rho_{L}<\rho_{R}$, then the Riemann problem between $(\rho_{L},u_{L})$ and $(\rho_{R}, u_{R})$ contains a $1$-shock . If $u_{L}>u_{R}$ and $\rho_{L}>\rho_{R}$, it contains a $2$-shock.
\end{lemma}
\begin{proof}
 Let us denote by $(\rho_{I}, u_{I})$ the intermediate state of this Riemann problem. As the velocity increases through rarefaction waves, the Riemann problem contains at least one shock. Moreover, the density decreases through $1$-rarefaction waves and $2$-shocks, and increases through $1$-shocks and $2$-rarefaction waves. Hence, if $\rho_{L}<\rho_{R}$, the Riemann problem contains a $1$-shock, while if $\rho_{L}>\rho_{R}$, it contains a $2$-shock.
\end{proof}
The scheme is built as follow.

\emph{Step $1$: Reconstruction}

\noindent We decide wether we should try to reconstruct in cell $j$ using Lemma~\ref{Detect}. We denote by $(\rho_{j}^{*}, u_{j}^{*})$ the intermediate state in the Riemann problem between $(\rho_{j-1}^{n}, u_{j-1}^{n})$ and $(\rho_{j+1}^{n}, u_{j+1}^{n})$. The candidates for the reconstructed states are:
\begin{equation} \label{DesRecStateL}
 (\bar{\rho}_{j,L}^{n}, \bar{u}_{j,L}^{n})= 
\begin{cases}
 (\rho_{j-1}^{n}, u_{j-1}^{n}) & \text{ if } u_{j-1}^{n}>u_{j+1}^{n} \text{ and } \rho_{j-1}^{n}< \rho_{j+1}^{n}; \\
 (\rho_{j}^{*}, u_{j}^{*}) & \text{ if } u_{j-1}^{n}>u_{j+1}^{n} \text{ and } \rho_{j-1}^{n}> \rho_{j+1}^{n}; \\
 (\rho_{j}^{n}, u_{j}^{n}) & \text{ otherwise; } 
\end{cases}
\end{equation}
 and
 \begin{equation} \label{DesRecStateR}
 (\bar{\rho}_{j,R}^{n}, \bar{u}_{j,R}^{n})= 
\begin{cases}
 (\rho_{j}^{*}, u_{j}^{*}) & \text{ if } u_{j-1}^{n}>u_{j+1}^{n} \text{ and } \rho_{j-1}^{n}> \rho_{j+1}^{n}; \\
 (\rho_{j+1}^{n}, u_{j+1}^{n}) & \text{ if } u_{j-1}^{n}>u_{j+1}^{n} \text{ and } \rho_{j-1}^{n}< \rho_{j+1}^{n}; \\
 (\rho_{j}^{n}, u_{j}^{n}) & \text{ otherwise. } 
\end{cases}
\end{equation}
It corresponds respectively to a $1$-shock reconstruction, a $2$-shock reconstruction and no reconstruction. 
\begin{rem}
 If by extraordinary $(\rho_{j-1}^{n}, u_{j-1}^{n})$ and $(\rho_{j+1}^{n}, u_{j+1}^{n})$ were linked by a single shock, we will have
 $$ (\bar{\rho}_{j,L}^{n}, \bar{u}_{j,L}^{n})= (\rho_{j-1}^{n}, u_{j-1}^{n}) \ \text{ and } \  (\bar{\rho}_{j,R}^{n}, \bar{u}_{j,R}^{n})= (\rho_{j+1}^{n}, u_{j+1}^{n}).$$
 This is crucial to approach exactly pure shocks.
\end{rem}
We denote by $\sigma_{j}$ the associated speed of the shock, set arbitrarily to $0$ when no reconstruction is performed:
\begin{equation} \label{Speed}
\sigma_{j}= 
\begin{cases}
 u_{j-1}^{n} - c \sqrt{\frac{\rho_{j}^{*}}{\rho_{j-1}^{n}}}  & \text{ if } u_{j-1}^{n}>u_{j+1}^{n} \text{ and } \rho_{j-1}^{n}> \rho_{j+1}^{n}; \\
 u_{j+1}^{n} + c \sqrt{\frac{\rho_{j}^{*}}{\rho_{j+1}^{n}}} & \text{ if } u_{j-1}^{n}>u_{j+1}^{n} \text{ and } \rho_{j-1}^{n}< \rho_{j+1}^{n}; \\
 0 & \text{ otherwise; } 
\end{cases}
\end{equation}
We then compute the distances to the left interface. We now have two distances, one for the conservation law on the density $\rho$ and one for the conservation law on the momentum $q=\rho u$:
 \begin{equation} \label{Dist}
d_{j}^{n,\rho}= \Delta x \frac{\bar{\rho}_{j,R}^{n}-\bar{\rho}_{j}^{n}}{\bar{\rho}_{j,R}^{n}-\bar{\rho}_{j,L}^{n}} \ \text{ and } \ d_{j}^{n,q}= \Delta x \frac{\bar{q}_{j,R}^{n}-\bar{q}_{j}^{n}}{\bar{q}_{j,R}^{n}-\bar{q}_{j,L}^{n}} .
\end{equation}
We insists on the fact that even though we reconstruct a different discontinuity in $\rho$ and $q$, it has the same speed $\sigma_{j}^{n}$ in the two variables.
We must now decide whether the reconstruction is accepted or not. The most natural choice is to accept the reconstruction when $0<d_{j}^{n,\rho}< \Delta x$ and $0<d_{j}^{n,q}<\Delta x$, i.e. when we are able to place the discontinuity inside the cell, such that both $\rho$ and $q$ are conserved in the cell. This choice will be referred to as the \emph{fully conservative reconstruction}. However, we also consider the case where the reconstruction is accepted whenever $0<d_{j}^{n,\rho}< \Delta x$. No condition is required on $d_{j,n}^{q}$, which means that the reconstruction on the momentum can be nonconservative. This choice will be referred to as the \emph{half conservative reconstruction}. The numerical tests below will show the interest of this less severe choice.
\begin{defi} For the isothermal Euler equation~\eqref{EulerEq}, the left and right reconstructed states are:
\begin{itemize}
 \item for the fully conservative reconstruction,
 \begin{equation} \label{RecStates1}
 (\rho_{j,L}^{n}, u_{j,L}^{n}, \rho_{j,R}^{n}, u_{j,R}^{n}) = 
\begin{cases}
  (\bar{\rho}_{j,L}^{n}, \bar{u}_{j,L}^{n}, \bar{\rho}_{j,R}^{n}, \bar{u}_{j,R}^{n}) & \text{ if } 0<d_{j}^{n,\rho}< \Delta x \text{ and } 0<d_{j}^{n,q}< \Delta x; \\
  (\rho_{j}^{n}, u_{j}^{n}, \rho_{j}^{n}, u_{j}^{n}) & \text{ otherwise. } 
\end{cases}
\end{equation}
 \item for the half conservative reconstruction,
 \begin{equation} \label{RecStates2}
 (\rho_{j,L}^{n}, u_{j,L}^{n}, \rho_{j,R}^{n}, u_{j,R}^{n}) = 
\begin{cases}
  (\bar{\rho}_{j,L}^{n}, \bar{u}_{j,L}^{n}, \bar{\rho}_{j,R}^{n}, \bar{u}_{j,R}^{n}) & \text{ if } 0<d_{j}^{n,\rho}< \Delta x;  \\
  (\rho_{j}^{n}, u_{j}^{n}, \rho_{j}^{n}, u_{j}^{n}) & \text{ otherwise. } 
\end{cases}
\end{equation}
\end{itemize}
where $\bar{\rho}_{j,L}^{n}, \bar{u}_{j,L}^{n}, \bar{\rho}_{j,R}^{n}$ and $\bar{u}_{j,R}^{n}$ are defined in~(\ref{DesRecStateL},\ref{DesRecStateR}), and $d_{j}^{n,\rho}$, $d_{j}^{n,q}$ are defined in~\eqref{Dist}.
\end{defi}

\emph{Step $2$: Computation of the fluxes}

\noindent The fluxes are computed exactly as in the previous section, Equations~\eqref{flux1/2} and~\eqref{flux-1/2}. The only difference is that in general, $d_{j}^{n, \rho} \neq d_{j}^{n,q}$ and we now have two crossing times
$$
\begin{cases}
 \Delta t_{j+1/2}^{\rho} &=  \frac{\Delta x - d_{j}^{n, \rho}}{\sigma_{j}-V_{\text{mesh}}}; \\[10pt]
  \Delta t_{j+1/2}^{q} &=  \frac{\Delta x - d_{j}^{n, q}}{\sigma_{j}-V_{\text{mesh}}}; 
\end{cases}
\ \ \ \text{ or } \ \ \ 
\begin{cases}
 \Delta t_{j-1/2}^{\rho}= \frac{d_{j}^{n, \rho}}{V_{\text{mesh}}-\sigma_{j}};\\[10pt]
 \Delta t_{j-1/2}^{q}= \frac{d_{j}^{n, q}}{V_{\text{mesh}}-\sigma_{j}},
\end{cases}
 $$
for $V_{\text{mesh}}<0$ and $V_{\text{mesh}}>0$ respectively. The fluxes are given by
\begin{equation} \label{Eflux1/2}
\left\{
\begin{aligned}
 \Delta t f_{j+1/2}^{ n, \rho}=  & \min(\Delta t, \Delta t_{j+1/2}^{\rho})  \left( f^{\rho}(\rho_{j,R}^{n}, q_{j,R}^{n}) - V_{\textrm{mesh}} \rho_{j,R}^{n} \right) \\
 & +  (\Delta t - \min(\Delta t, \Delta t_{j+1/2}^{\rho})) \left( f^{\rho}(\rho_{j,L}^{n}, q_{j,L}^{n}) - V_{\textrm{mesh}} \rho_{j,L}^{n} \right); \\
  \Delta t f_{j+1/2}^{ n, q}=  & \min(\Delta t, \Delta t_{j+1/2}^{q})  \left( f^{q}(\rho_{j,R}^{n}, q_{j,R}^{n}) - V_{\textrm{mesh}} q_{j,R}^{n} \right) \\
 & +  (\Delta t - \min(\Delta t, \Delta t_{j+1/2}^{q})) \left( f^{q}(\rho_{j,L}^{n}, q_{j,L}^{n}) - V_{\textrm{mesh}} q_{j,L}^{n} \right),
\end{aligned}
\right.
\end{equation}
when $V_{\text{mesh}}$ is positive, and
\begin{equation} \label{Eflux-1/2}
\left\{
\begin{aligned}
\Delta t f_{j-1/2}^{n, \rho}=  & \min(\Delta t, \Delta t _{j-1/2}^{\rho}) \left( f^{\rho}(\rho_{j,L}^{n}, q_{j,L}^{n})- V_{\textrm{mesh}} \rho_{j,L}^{n} \right) \\
 & +  (\Delta t - \min(\Delta t, \Delta t _{j-1/2}^{\rho})) \left( f^{\rho}(\rho_{j,R}^{n}, q_{j,R}^{n}) - V_{\textrm{mesh}} \rho_{j,R}^{n} \right); \\
 \Delta t f_{j-1/2}^{n, q}=  & \min(\Delta t, \Delta t _{j-1/2}^{q}) \left( f^{q}(\rho_{j,L}^{n}, q_{j,L}^{n})- V_{\textrm{mesh}} q_{j,L}^{n} \right) \\
 & +  (\Delta t - \min(\Delta t, \Delta t _{j-1/2}^{q})) \left( f^{q}(\rho_{j,R}^{n}, q_{j,R}^{n}) - V_{\textrm{mesh}} q_{j,R}^{n} \right); 
\end{aligned}
\right.
\end{equation}
when $V_{\text{mesh}}$ is negative. Here, $f^{\rho}$ and $f^{q}$ denote the two components of the flux: $f^{\rho}(\rho, q)= q$ and $f^{q}(\rho, q)= \frac{q^{2}}{\rho} + c^{2} \rho$.

\vspace{3mm}

\emph{Step $3$: March in time}

\noindent Eventually, the conservative variables are updated to the next time step:
\begin{equation} 
\begin{cases}
 \rho_{j}^{n+1}= \rho_{j}^{n} - \frac{\Delta t}{\Delta x} ( f_{j+1/2}^{n,\rho} -  f_{j-1/2}^{n,\rho} ), \\
 q_{j}^{n+1}= q_{j}^{n} - \frac{\Delta t}{\Delta x} ( f_{j+1/2}^{n,q} -  f_{j-1/2}^{n,q} ) .
\end{cases}
\end{equation}

\begin{prop} \label{PureShock2}
 The reconstruction scheme is consistant, and exact on pure shocks. 
 \end{prop}
\begin{proof} We proceed once again by induction, and focus on the case of a $1$-shock. We denote by $(\rho_{L}, q_{L})$ (resp. $(\rho_{R}, q_{R})$) the left (resp. the right) density and momentum. Suppose that at the $n$-th iteration, the scheme gave the exact average of the solution. Denote by $j_{0}$ the cell where the shock lies, and by $\delta$ its distance to the right interface of the $j_{0}$-th cell. In other words, we have
$$
\begin{cases}
 \rho_{j}^{n}= \rho_{L} & \mbox{ if $j<j_{0}$}, \\
 \rho_{j_{0}}^{n}= \frac{\delta}{\Delta x} \rho_{L}  + \frac{\Delta x -\delta}{\Delta x} \rho_{R}  & \mbox{ if $j=j_{0}$}, \\
 \rho_{j}^{n}= \rho_{R} & \mbox{ if $j>j_{0}$}, \\
\end{cases}
\ \ \text{ and } \ \
\begin{cases}
 q_{j}^{n}= q_{L} & \mbox{ if $j<j_{0}$}, \\
 q_{j_{0}}^{n}= \frac{\delta}{\Delta x} q_{L}  + \frac{\Delta x -\delta}{\Delta x} q_{R}  & \mbox{ if $j=j_{0}$}, \\
 q_{j}^{n}= q_{R} & \mbox{ if $j>j_{0}$}. \\
\end{cases}
$$
Let us first check that $u_{j_{0}-1}^{n} \geq u_{j_{0}}^{n} \geq u_{j_{0}+1}^{n}$. This insures that a $1$-shock is detected in cell $j_{0}$ with Lemma~\ref{Detect}. We denote by $\beta$ the quantity $\frac{\alpha}{\Delta x}$. We have
\begin{align*}
  u_{j_{0}} - u_{L} &= \frac{\beta q_{L} + (1-\beta) q_{R}}{\beta \rho_{L}+ (1-\beta)\rho_{R}} - \frac{q_{L}}{\rho_{L}} \\
  	& =\frac{\rho_{L}(\beta q_{L} + (1-\beta) q_{R})-q_{L}(\beta \rho_{L}+ (1-\beta)\rho_{R})}{\rho_{L}(\beta \rho_{L}+ (1-\beta)\rho_{R})} \\
	& =\frac{(1-\beta) (\rho_{L}q_{R}-\rho_{R} q_{L})}{\rho_{L}(\beta \rho_{L}+ (1-\beta)\rho_{R})} \\
	& =\frac{(1-\beta)\rho_{L} \rho_{R} (u_{R}-u_{L})}{\rho_{L}(\beta \rho_{L}+ (1-\beta)\rho_{R})} \\
	& \leq 0,
\end{align*}
and we obtain similarly that $u_{j_{0}}$ is bigger than $u_{R}$. Therefore the $1$-shock is detected, and with no change compared to the scalar case, its position and speed are correctly reconstructed. It remains to prove that no reconstruction is performed on cells $j_{0}-1$ and $j_{0}+1$. In both those cells, Lemma~\ref{Detect} detects a $1$-shock. Therefore in cell $j_{0}-1$, the left reconstructed density is $\rho_{L}$ and, by conservation of $\rho$, no reconstruction can be performed in this cell. In cell $j_{0}+1$, no reconstruction is performed if the Riemann problem between $(\rho_{j_{0}}, q_{j_{0}})$ and $(\rho_{R}, q_{R})$ contains a $1$-shock and a $2$-rarefaction. Indeed in that case, the density is increasing, and both the left and the right reconstructed densities are smaller than $\rho_{j_{0}+1}= \rho_{R}$, which blocks the reconstruction. Let us prove that this Riemann problem cannot contains a $1$-shock and a $2$-shock. If it was possible, the middle state $(\rho_{*},q_{*})$ would belong to the light gray region in Figure~\ref{Fnoreconstruction}. Indeed, we would have $u_{*}>u_{R}$ and $\rho_{*}>\rho_{R}$. %Moreover, $(\rho_{L}, q_{L})$ lies in the dark gray area, because through a $1$-shock, both density and velocity increase. 
On the other hand, $(\rho_{*}, q_{*})$ is also on the $1$-wave curve of $(\rho_{j_{0}}, q_{j_{0}})$. This state lies on the chord joining $(\rho_{L}, q_{L})$ and $(\rho_{R}, q_{R})$, and hence the $1$-wave curve  of $(\rho_{j_{0}}, q_{j_{0}})$ (dark gray curve in the Figure~\ref{Fnoreconstruction}) is below the $1$-wave curve  of $(\rho_{L}, q_{L})$. This last curve does not pass through the light gray area, because it is concave and its slope in $\rho_{R}$ is smaller than $u_{R}$, and we obtain a contradiction.
%
%Therefore, if there were a $1$-shock between $(\rho_{j_{0}},q_{j_{0}})$ and $(\rho_{*},q_{*})$ and a $2$-shock between $(\rho_{*},q_{*})$ and $(\rho_{R},q_{R})$, there would exist an intersection point between the $1$-Lax-curve starting from $(\rho_{L},q_{L})$ and the  $1$-Lax-curve starting from $(\rho_{j_{0}},q_{j_{0}})$. It yields that $(\rho_{j_{0}},q_{j_{0}})$ lies on the $1$-Lax-curve starting from $(\rho_{L},q_{L})$, which contradicts its concavity.
\begin{psfrags} 
 \psfrag{rL}{$(\rho_{L}, q_{L})$}
 \psfrag{rI}{$(\rho_{j_{0}}, q_{j_{0}})$}
 \psfrag{rR}{$(\rho_{R}, q_{R})$}
 \psfrag{r*}{$(\rho_{*}, q_{*})$}
 \psfrag{r}{$\rho$}
 \psfrag{q}{$q$}
 \begin{figure}[H]
 \centering
 \includegraphics[width=10cm]{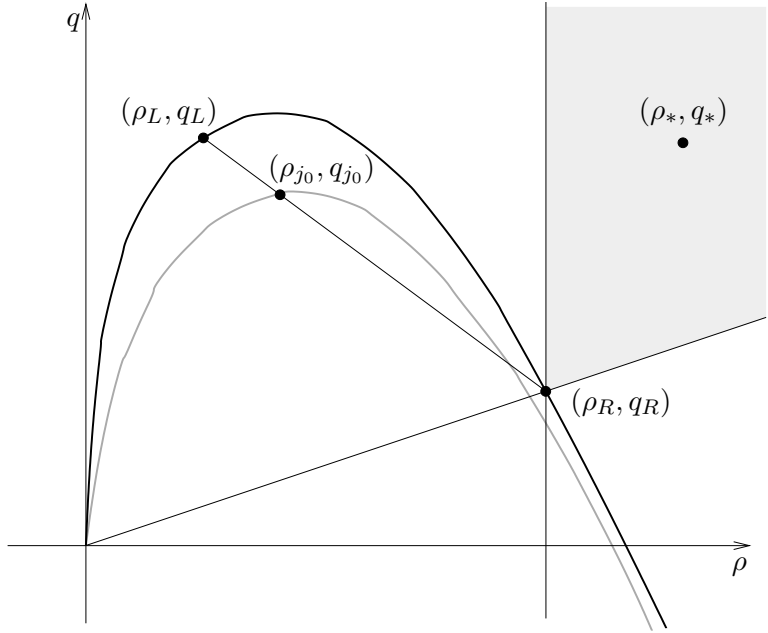}
 \caption{The Riemann problem between cells $j_{0}$ and $j_{0}+2$ cannot consist of a $1$-shock and a $2$-shock}  \label{Fnoreconstruction}
 \end{figure}
\end{psfrags}
Once we notice that at each time step, the two reconstructed distances $d_{j}^{n,\rho}$ and $d_{j}^{n,q}$ are equal, the remaining of the proof is identical to that of the scalar case.
 \end{proof}
 
 \subsection{Slowly moving shocks}
Many finite volume schemes fails to approach correctly shocks that move slowly compared to the speed of sound. Typically, a spike in the momentum appears in the first iterations in time, and is then, by conservation of the momentum, counterbalance by a hollow. The spike oscillates through time in an almost periodic manner, where the period corresponds to the 
 time that the shock needs to cross an entire cell. Therefore, even though the oscillations are diffused by the scheme, it is a constant source of error that is blamed for slow convergence to the steady state (see~\cite{N87} and~\cite{M94}). Moreover, higher order schemes tend to better preserve the spurious oscillations (that should not be here) better than first order scheme. This problem was first report by Colella and Woodward in~\cite{CW84}. In~\cite{R90} and~\cite{AR97}, numerous numerical tests and comparison between schemes are performed. Unlike all schemes tested in those papers, the half and fully conservative  reconstruction schemes are exact not only on steady shocks, but on all shocks (see Proposition~\ref{PureShock2}). This can be checked on the top of Figure~\ref{FSMS}, where the spurious oscillations created by \emph{all} the other schemes can be seen.  This good property is inherited when computing a Riemann problem containing a slowly moving shock with the half conservative reconstruction scheme, but lost with the fully conservative reconstruction scheme, as depicted on the bottom of Figure~\ref{FSMS}. Those simulations are run on a mesh with $200$ cells, with a CFL number of $0.45$. In both simulations, $\rho_{L}=1$ and $\rho_{R}=20$ and the speed of sound is $c=0.5$. 
\begin{figure}[H]
\centering
\includegraphics[width=16cm]{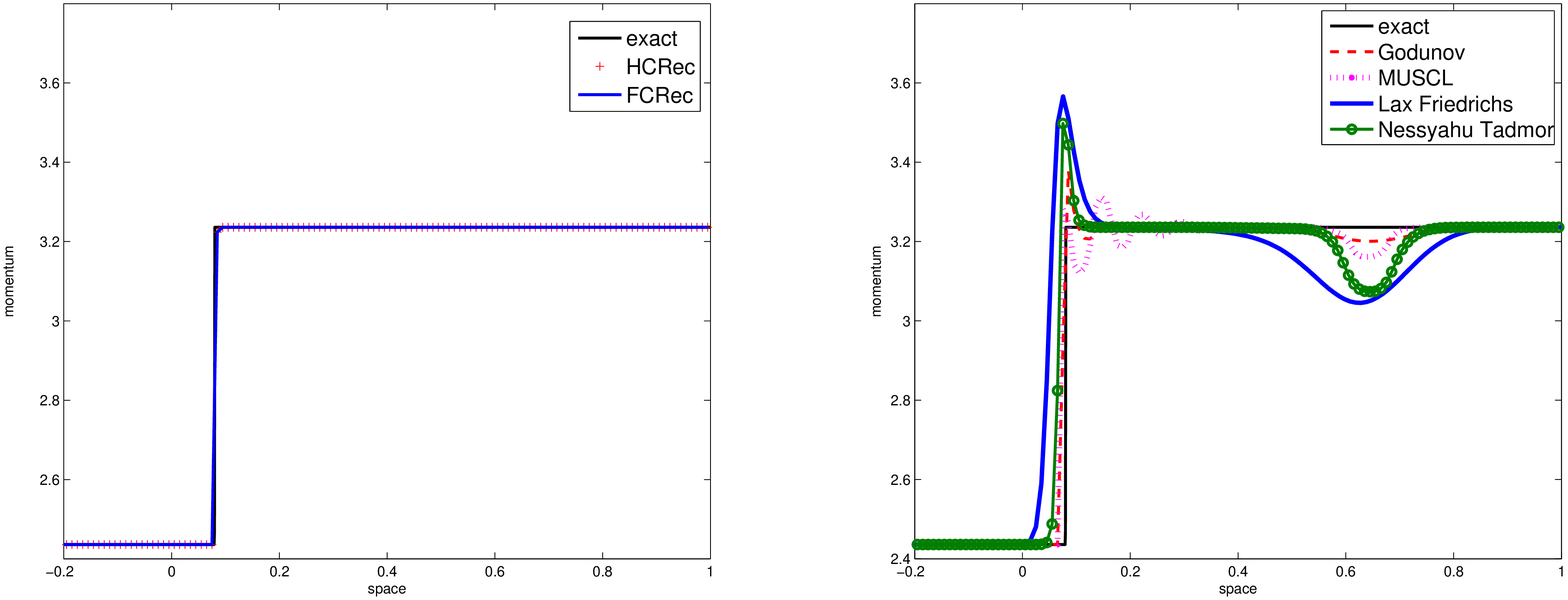}
\includegraphics[width=16cm]{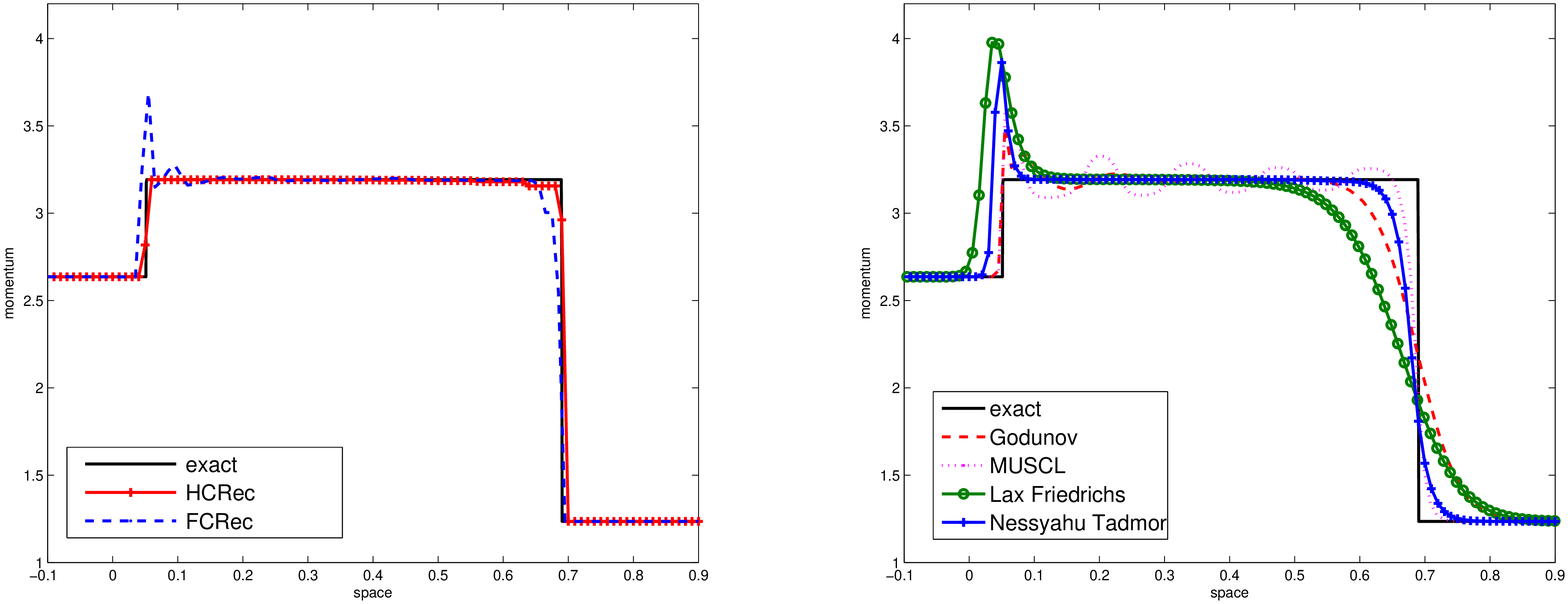}
\caption{Comparison of the half conservative (HCRec) and fully conservative (FCRec) reconstruction schemes with classical schemes on a unique slowly moving shock (top) and on a Riemann problem consisting of a slowly moving $1$-shock followed by a $2$-shock (bottom).}  \label{FSMS}
\end{figure}
The data on top of Figure~\ref{FSMS} are prepared to correspond to a $1$-shock moving at speed $0.1$: $u_{L}=0.1+ c*\sqrt{ \frac{\rho_{R}}{\rho_{L}}}$ and $u_{R}=0.1+ c*\sqrt{ \frac{\rho_{L}}{\rho_{R}}}$, which yields $q_{L} \approx 2.4361$ and $q_{R} \approx 3.2361$. It crosses a cell in approximately $6.25$ iterations. Inthe bottom of Figure~\ref{FSMS}, the momentum is modified: $q_{L} \approx 2.6361$ and $q_{R} \approx 1.2361$. When the initial data consist of a Riemann problem only containing shocks, even if one of them is a slowly moving one, the reconstruction scheme is of order $2$. This is illustrated on the left of Figure~\ref{FOrder}. On the right of this figure, we can see that the error is entirely due to the first iterations in time. Once the two shocks are separated, the shocks are better approximated, and no diffusion appears. The data are the same than in the bottom of Figure~\ref{FSMS}.
 \begin{figure}[H]
\centering
\includegraphics[width=7cm]{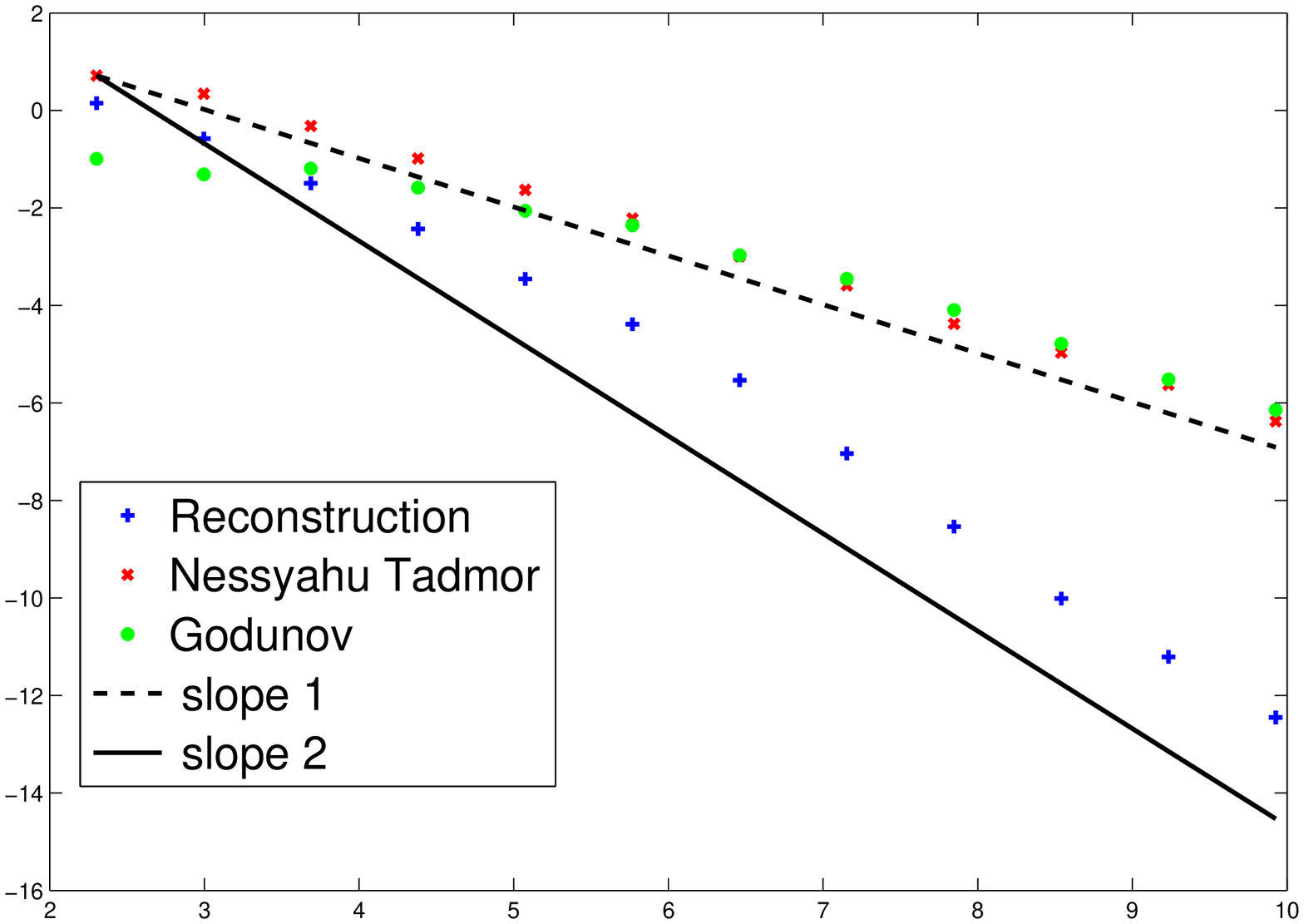}
\includegraphics[width=7cm]{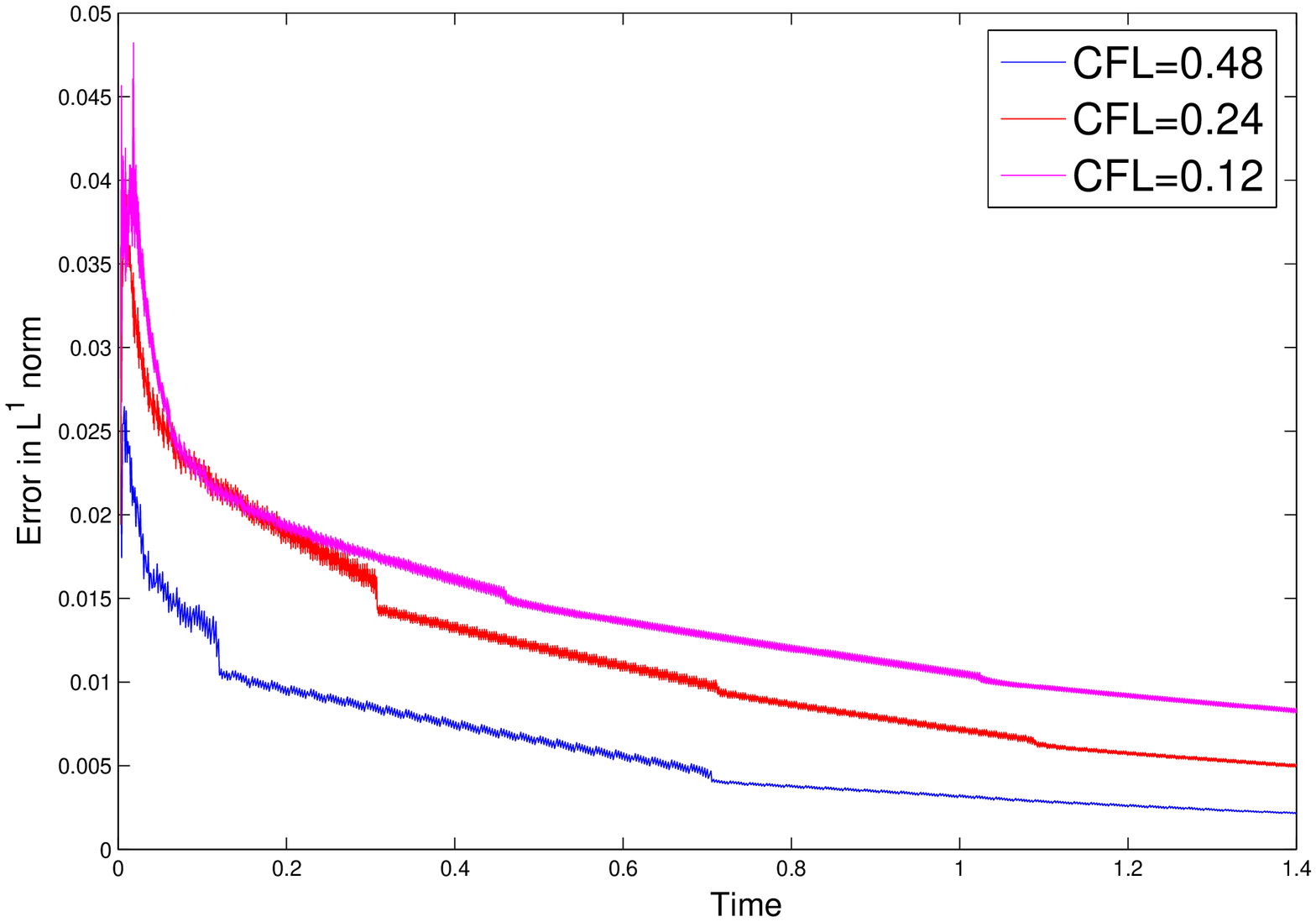}
\caption{Left: order of the half conservative reconstruction on the Riemann problem only containing shocks. Right: $L^{1}$ error through time for different CFL conditions, for a mesh of $500$ cells.}  \label{FOrder}
\end{figure}

An explanation of the superiority of the half conservative reconstruction scheme over the other schemes (including the fully conservative reconstruction scheme) might be linked to the study of slowly moving shocks by Jin and Liu in~\cite{JL96}. In this paper, they proved with a traveling wave analysis that a momentum spike appears in the viscous shock profile for the Euler equations with a linear viscosity
 $$ 
\begin{cases}
 \partial_{t} \rho + \partial_{x} (\rho u) = \eps \partial_{xx} \rho, \\
 \partial_{t}(\rho u) + \partial_{x} (\rho u^{2} + c^{2} \rho) = \eps \partial_{xx} q,
\end{cases}
 $$
 when the density has a monotonous profile. When the viscosity is Navier--Stokes-like, the momentum remains monotonous . We believe that the half conservative reconstruction scheme is the only one that has a numerical viscosity looking like the physical Navier-Stokes' viscosity. Indeed, as the reconstruction in the density is performed whenever it is possible (while there is an additional constraint on the momentum in the fully conservative reconstruction), the numerical viscosity is likely to be zero on the mass conservation law as often as possible. Another argument is that if the momentum spike appears in the fully conservative reconstruction scheme, it blocks the reconstruction and the scheme will later behave like the Lax-Friedrichs scheme, while the half conservative reconstruction scheme is more flexible, and will continue to reconstruct near the shock.

\subsection{Use of other schemes in the rarefaction waves}
When no shocks are detected, or when the reconstruction is not accepted, the flux degenerates toward the Lax-Friedrichs flux. As a consequence, it is very diffusive inside the rarefaction waves. A simple cure is to the replace, whenever it is chosen, the Lax-Friedrichs flux by a more accurate one. The fluxes write (for the half conservative scheme):
$$ f_{j+1/2}^{n} = 
\begin{cases}
 f_{j+1/2}^{n, REC} & \text{ if } ( 0 < d_{j}^{n} < dx \text{ and } V_{\text{mesh}}^{n}<0 ) \text{ or } ( 0 < d_{j+1}^{n} < dx \text{ and } V_{\text{mesh}}^{n}>0 ),  \\
 f_{j+1/2}^{n, NT} & \text{otherwise},
\end{cases} $$
where $f_{j+1/2}^{n,REC}$ is defined by~\eqref{Eflux1/2} and~\eqref{Eflux-1/2}, while $f_{j+1/2}^{n,NT}$ is another flux on a staggered grid.  In our implementation we use the simplest version of the Nessyahu and Tadmor scheme~\cite{NT90}. There is no Riemann problems to solve, so it almost does not impact the computation time. As expected, this coupled scheme behaves like the Nessyahu--Tadmor scheme in the rarefaction wave, and like the reconstruction scheme on the shock. This is illustrated by the Figure~\ref{FCoupling} below.
\begin{figure}[H]
\centering
\includegraphics[width=14cm]{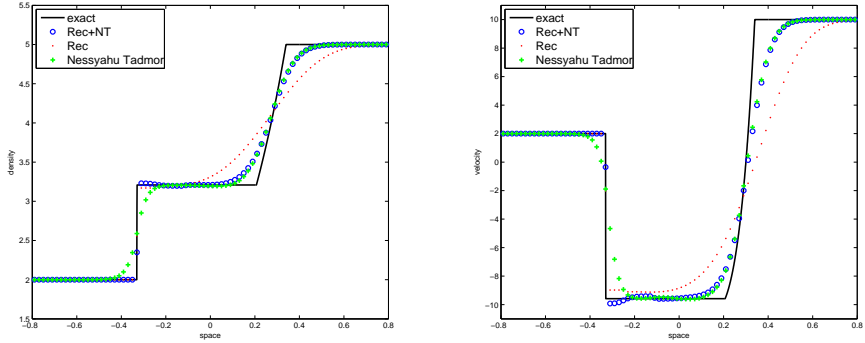}
\caption{Comparison of the reconstruction scheme (Rec), the Nessyahu--Tadmor scheme and the coupled scheme (Rec+NT) on a Riemann problem consisting of a $1$-shock and a $2$-rarefaction wave. The mesh has $100$ cells and the CFL number is $0.1$.}  \label{FCoupling}
\end{figure}

\section{Prospective results on the full Euler equations}

Let us now consider the full gas dynamics equation for an ideal gas:
$$ 
\begin{cases}
 \partial_{t} \rho + \partial_{x} (\rho u) = 0, \\
 \partial_{t}(\rho u) + \partial_{x} (\rho u^{2} + p )=0, \\
 \partial_{t} E + \partial_{x}(u(E+p))=0,
\end{cases}
$$
where $\rho$ is the density, $u$ is the fluid's velocity and $E$ is the total energy per unit volume. The pressure $p$ writes
$$ p = p(\rho, e) = e (\gamma-1) \rho, $$
where $e=\frac{E}{\rho}- \frac{1}{2} u^{2}$ is the specific internal energy and $\gamma$ is the ratio of specific heat. The complete solution of the Riemann problem can be found, for example, in~\cite{Toro}. The novelty is that the $2$-wave is a contact discontinuity. We denote by $U_{j}^{*, L/R}=(\rho_{j}^{*,L}, \rho_{j}^{*,R}, u_{j}^{*}, p_{j}^{*})$ the left and right density, velocity and pressure appearing around the contact discontinuity in the Riemann problem between $U_{j-1}=(\rho_{j-1}^{n}, u_{j-1}^{n}, p_{j-1}^{n})$ and  $U_{j+1}=(\rho_{j+1}^{n}, u_{j+1}^{n}, p_{j+1}^{n})$. Let us denote by $C_{CFL}$ the CFL number. The wave to be reconstruct is chosen as follow. 
\begin{itemize}
 \item If $u_{j-1}^{n} \geq u_{j+1}^{n}$, $\rho_{j-1}^{n} \leq \rho_{j+1}^{n}$, $p_{j-1}^{n} \leq  p_{j+1}^{n}$ and if
 $$ |\rho_{j-1}-\rho_{j}^{*,L}|> C_{CFL} \max( |\rho_{j}^{*,L}-\rho_{j}^{*,R}|,  |\rho_{j}^{*,R}-\rho_{j+1}|), $$
 we set the desired reconstructed states to be
 $$ \bar{U}_{j,L}= U_{j-1} \ \ \ \text{ and } \ \ \ \bar{U}_{j,R}= U_{j}^{*,L}. $$ 
 In other words, we will try to reconstruct a $1$-shock.
 \item If $u_{j-1}^{n} \geq u_{j+1}^{n}$, $\rho_{j-1}^{n} \geq \rho_{j+1}^{n}$, $p_{j-1}^{n} \geq  p_{j+1}^{n}$, and if
  $$ |\rho_{j}^{*,R}-\rho_{j+1}|> C_{CFL} \max( |\rho_{j}^{*,L}-\rho_{j}^{*,R}|,  |\rho_{j-1}-\rho_{j}^{*,L}|), $$
we set the desired reconstructed states to be
 $$ \bar{U}_{j,L}= U_{j}^{*,R} \ \ \ \text{ and } \ \ \ \bar{U}_{j,R}= U_{j+1}. $$ 
 In other words, we will try to reconstruct a $3$-shock.
 \item Otherwise, and if
  $$ |\rho_{j}^{*,L}-\rho_{j}^{*,R}|> C_{CFL} \max( |\rho_{j-1}-\rho_{j}^{*,L}|,  |\rho_{j}^{*,R}-\rho_{j+1}|), $$
  we set
  $$ \bar{U}_{j,L}= U_{j}^{*,L} \ \ \ \text{ and } \ \ \ \bar{U}_{j,R}= U_{j}^{*,R}. $$ 
and try to reconstruct as a $2$-contact discontinuity.
\end{itemize}
Then, we compute the distances $d_{j}^{\rho}$, $d_{j}^{q}$ and $d_{j}^{E}$ by conservation of mass, momentum and total energy inside the $j$-th cell. The reconstruction is accepted when the following conditions are fulfilled:
\begin{itemize}
 \item Both $d_{j}^{\rho}$ and $d_{j}^{E}$ are between $0$ and $\Delta x$;
 \item Both triplets $(\rho_{j-1}^{n}, \rho_{j}^{n}, \rho_{j+1}^{n})$ and $(u_{j-1}^{n}, \frac{1}{\Delta x} \int_{x_{j-1/2}^{n}}^{x_{j+1/2}^{n}} u_{rec}(x) dx, u_{j+1}^{n})$ are monotonous;
 \item Eventually, $e_{rec}$ remains positive on the cell.
\end{itemize}
Here, $u_{rec}$ and $e_{rec}$ are the piecewise constant reconstructed velocities and internal energy on the $j$-cell. The reconstruction scheme appears to be much more complicated in this case. This is due to the lack of criterion to detect a dominant contact discontinuity. Indeed, Lemma~\ref{Detect} (even with an additional test on the pressure) is still valid for this system as $p$ and $u$ remains constant through contact discontinuities. But it is not entirely satisfactory to decide which wave prevails. Indeed, plenty Riemann problems with a $1$-shock and a $2$-contact discontinuity verify the inequalities
$$u_{L}^{n} \geq u_{R}^{n}, \ \rho_{L}^{n} \leq \rho_{R}^{n} \ \ \ \text{ and } \ \ \ p_{L}^{n} \geq  p_{R}^{n}$$
even though the shock is small and the rarefaction is strong. This is why we added additional constraints like
$$ |\rho_{j-1}-\rho_{j}^{*,L}|> C_{CFL} \max( |\rho_{j}^{*,L}-\rho_{j}^{*,R}|,  |\rho_{j}^{*,R}-\rho_{j+1}|), $$
with is a criterion on the force of the wave. Note that with this criterion, some shocks are not detected, while without it, some contact discontinuities are detected as shocks. We observe numerically that using the CFL number tunes the scheme pretty well. Instead of this criterion, we also tried a ``two shots reconstruction'', in which we first try to reconstruct everywhere with contact discontinuities. If this reconstruction fails (i.e. if we do not have $0<d_{j}^{\rho}<\Delta x $ and $0<d_{j}^{E}<\Delta x $), then we attempt to reconstruct $1$- and $3$-shocks using Lemma~\ref{Detect}. The numerical results are very similar to those presented below.This very same difficulty to detect contact discontinuities led us to add more constraints to accept a reconstruction. The first one is similar to the previous $2\times 2$ system. We do not impose a conservation constraint on the momentum to mimic the Navier Stockes viscosity. The density and velocity are monotonous along the profile of a viscous shock, which justify the second constraint. Eventually, the last constraint forces the pressure to remain positive.

We perform various numerical tests to compare the reconstruction scheme with other classical schemes. On all these figures, \texttt{Rec} indicates the half conservative reconstruction scheme and \texttt{Rec+NT} indicates the half conservation reconstruction scheme coupled with the Nessyahu-Tadmor scheme. Theses schemes are compared with the Godunov and Rusanov schemes, and with the simplest version of the Nessyahu-Tadmor scheme~\cite{NT90} and the simplest version of the MUSCL scheme~\cite{VL79} 
\begin{example} On Figure~\ref{Fcase8}, we compare numerous schemes on a Riemann problem with 
$$
\begin{cases}
 \rho_{L}=5.99924 \ u_{L}=19.5975  \ \ \text{ and } \ \ \ p_{L}=460.894, \\
 \rho_{R}=5.99242 \  u_{R}=-6.19633  \ \ \text{ and } \ \ \ p_{R}=46.0950.
\end{cases}
 $$
The solution consists of three discontinuities moving to the right (cf the book of Toro~\cite{Toro}). We took a CFL number of $0.4$ and discretized the interval in $400$ cells. The shocks are very well computed, an we observe an improvement on the contact discontinuity. This improvement is often much better when the reconstruction scheme is coupled with a higher order scheme. 
 \begin{figure}[H]
\centering
\includegraphics[width=16cm]{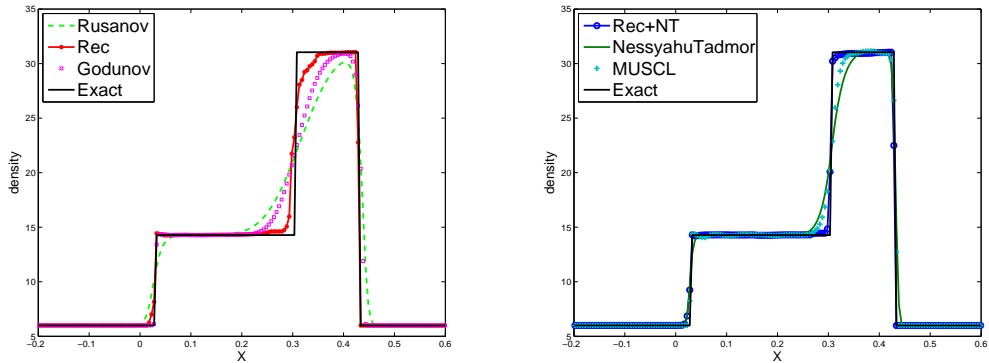}
\caption{Density at time $0.035$ for the first test case.}  \label{Fcase8}
\end{figure}
\end{example}

\begin{example}
 We compare those same schemes in the Colella and Woodward blast wave test case, introduced in~\cite{CW84}. The initial datum is
$$ 
\begin{cases}
 \rho^{0}(x)= 1; \\
u^{0}(x)=0;\\
p^{0}(x)= 1 \, 000* \mathbf{1}_{x<0.1}+0.01* \mathbf{1}_{0.1 <x < 0.9}+ 100*\mathbf{1}_{x\geq0.9} \, .
\end{cases}
$$
The solution is computed on the interval $[0,1]$, with reflective boundary conditions at the two extremities of the interval. The reference solution is obtained by running the Nessyahu Tadmor scheme on $30 \,  000$ cells. A more accurate reference solution can be found in~\cite{CW84}. On Figures~\ref{FCW1}, we used a mesh containing $400$ cells, a CFL number set to $0.45$ and the final time is $T=0.026$. We plot the density, velocity and internal energy. The non-dissipative feature of the reconstruction scheme is particularly obvious on this last plot. On Figure~\ref{FCW2}, the finite time is $T=0.038$. We took a CFL number of $0.48$ and $400$, $1 \, 200$ and $ 2 \, 000$ cells. Observe that even with few points, the discontinuities are sharply captured, even though no reconstruction is performed in the middle area, where we recover the behavior of Lax Friedriechs' or Nessayahu-Tadmor's schemes.
 \begin{figure}[H]
\centering
\includegraphics[width=16cm]{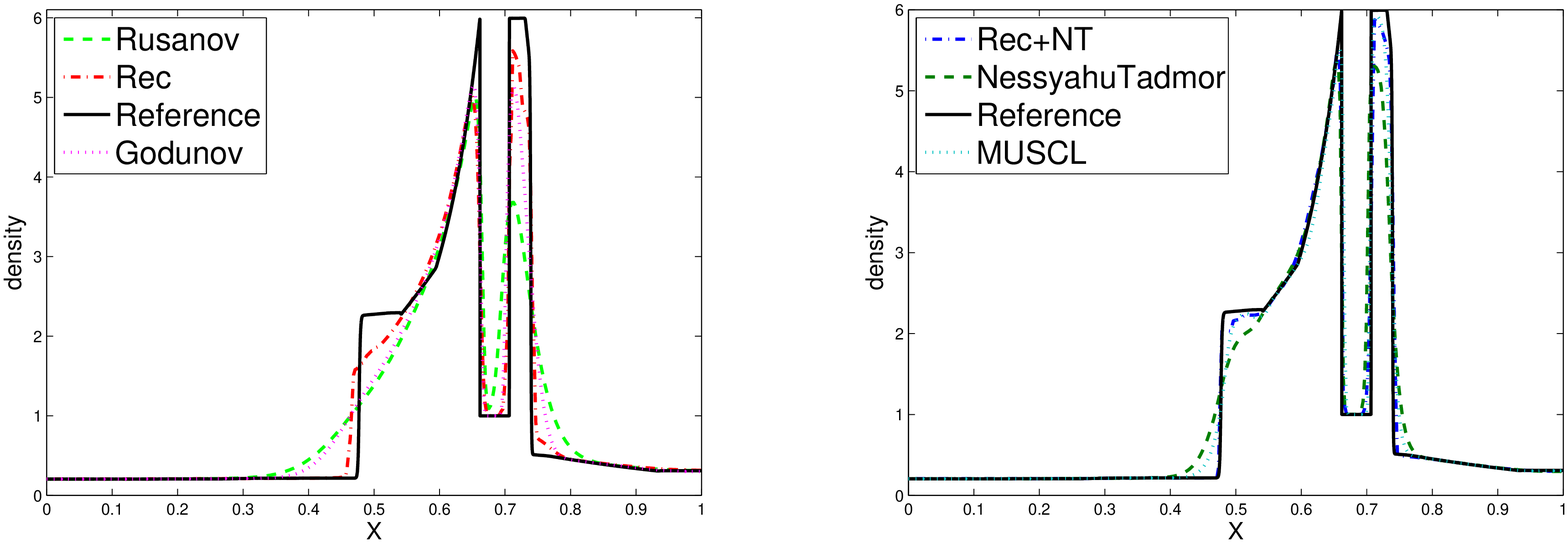}
\includegraphics[width=16cm]{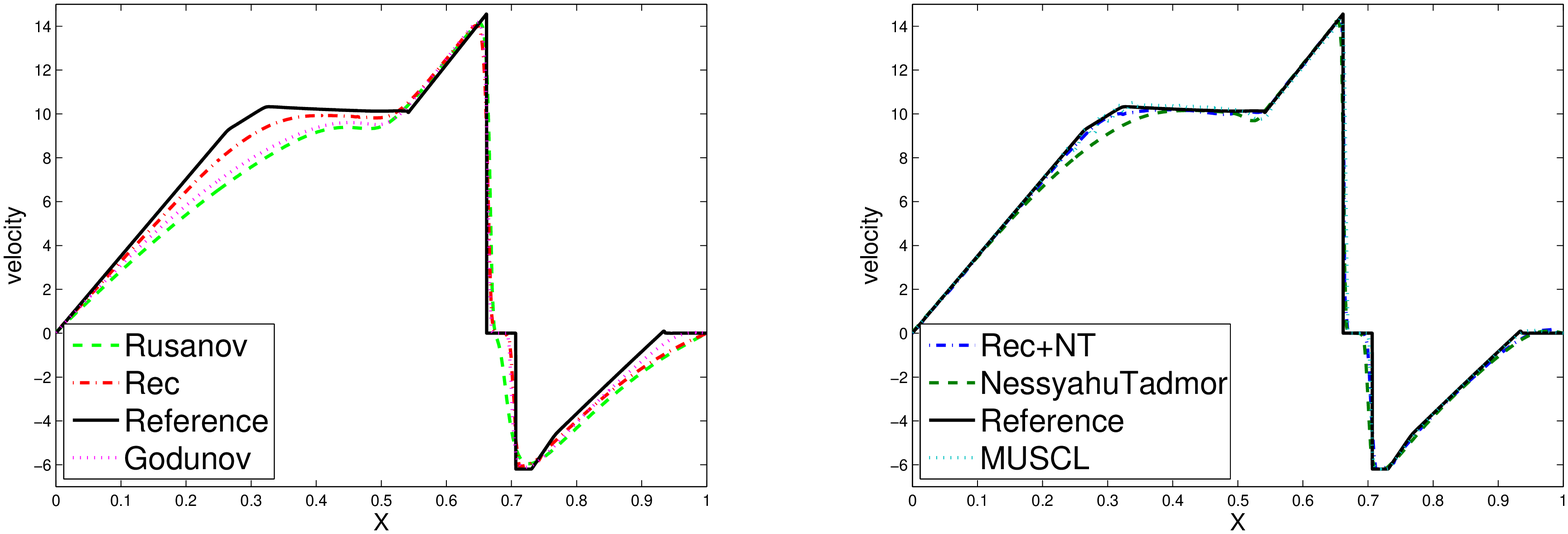}
\includegraphics[width=16cm]{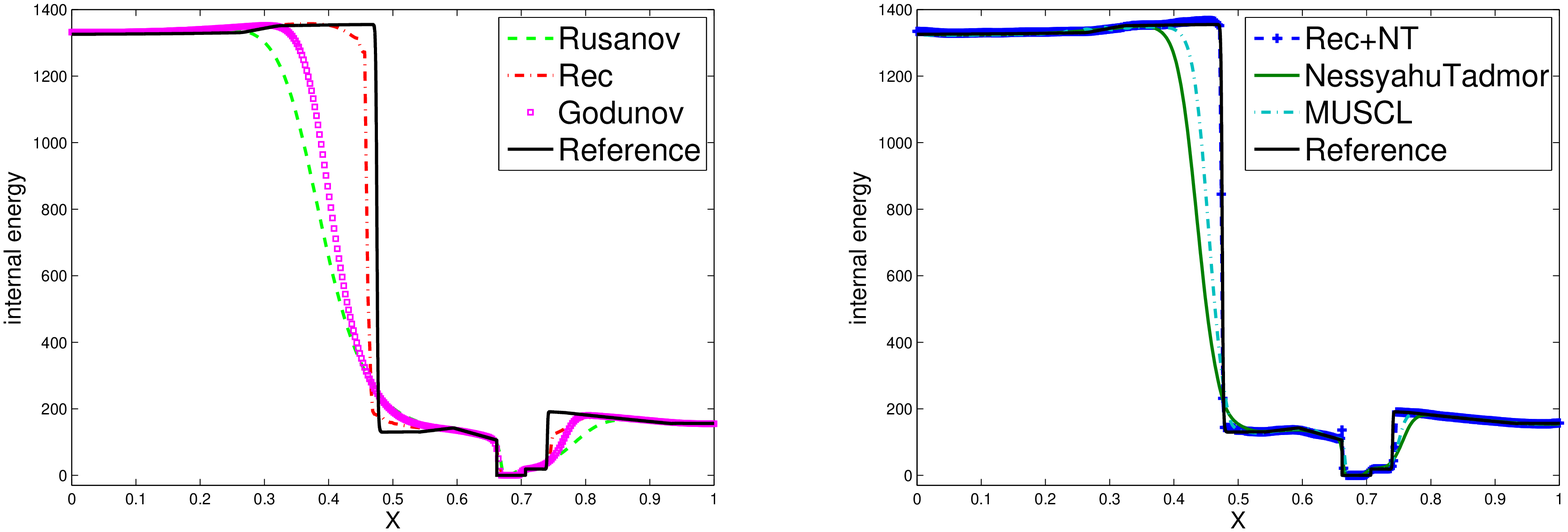}
\caption{From top to bottom, density, velocity and internal energy at time $0.026$ in the blast wave test case, with a mesh of $400$ cells.}  \label{FCW1}
\end{figure}

\begin{figure}[H]
\centering
\includegraphics[width=16cm]{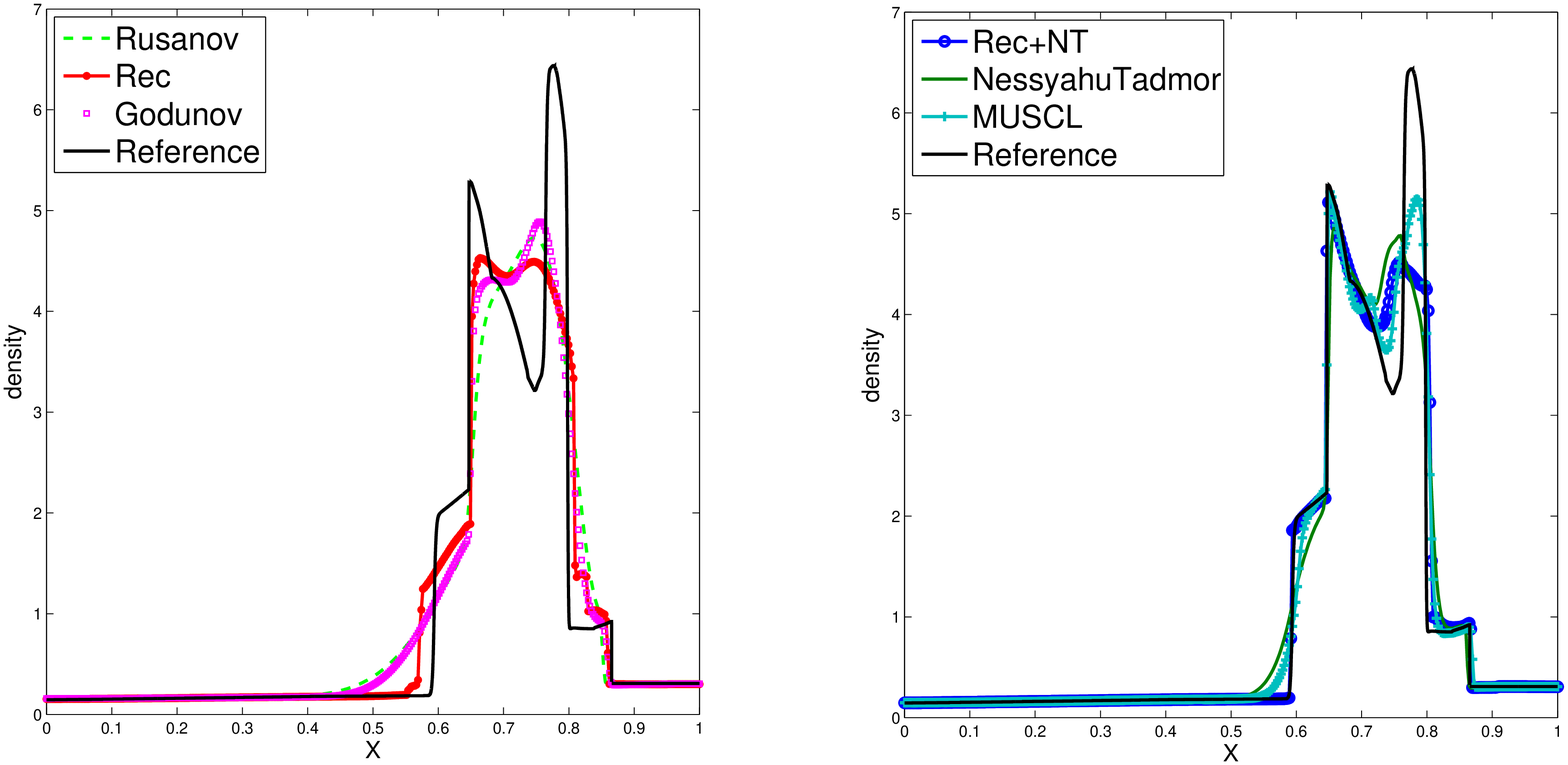}
\includegraphics[width=16cm]{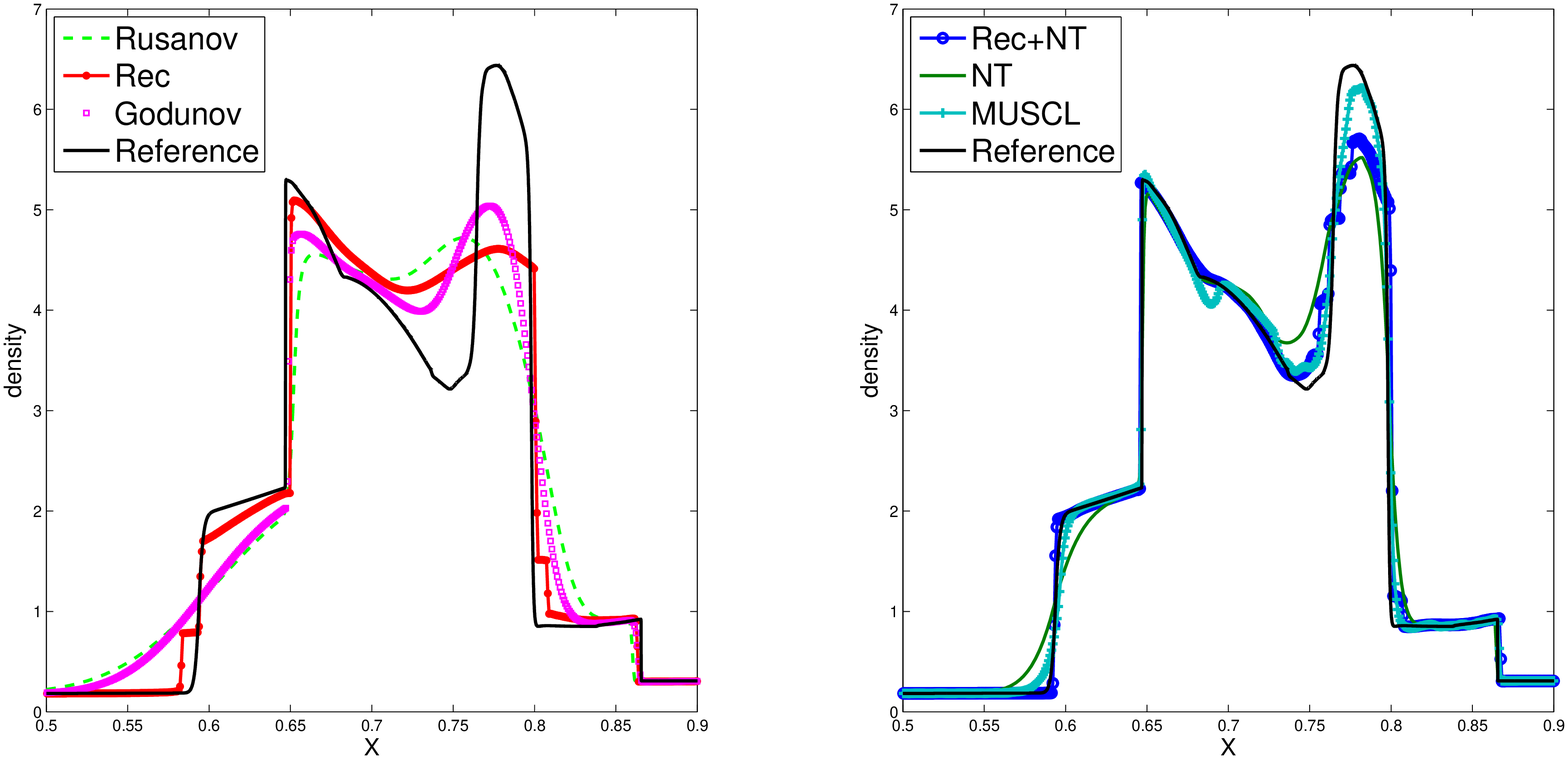}
\includegraphics[width=16cm]{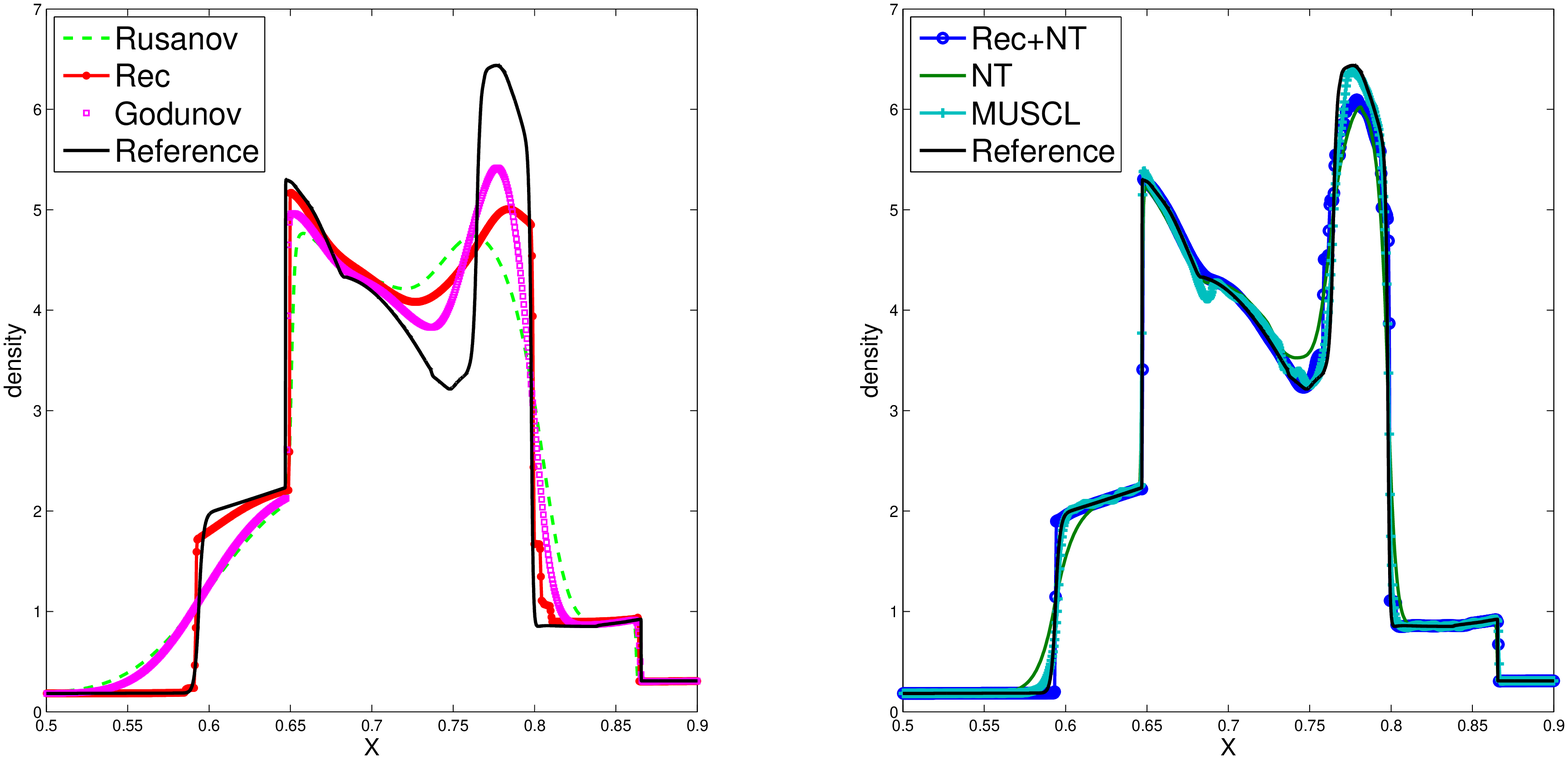}
\caption{Density at time $0.038$ in the blast wave test case, with a mesh of $400$ cells (top), $1 \, 200$ cells (middle) and $2 \, 000$ cells (bottom).}  \label{FCW2}
\end{figure}
\end{example} 
\begin{example} 
The Figure~\ref{FSine} presents the results at time $1.8$ of an entropy satisfying shock interacting with a sine wave, presented for example in~\cite{LW03}. The mesh has $400$ cells, the CFL number is $0.45$ and the initial datum is:
$$ 
\begin{cases}
 \rho^{0}(x)= 3.897143* \mathbf{1}_{x<-4}+(1+0.2\sin(5x)) \mathbf{1}_{x\geq-4} \, ; \\
u^{0}(x)= 2.629369* \mathbf{1}_{x<-4} \, ;\\
p^{0}(x)= 10.33333* \mathbf{1}_{x<-4}+ \mathbf{1}_{x\geq-4} \, .
\end{cases}
$$
The reference solution is, once again, the result given by the Nessyahu-Tadmor scheme with $30 \,  000$ cells and a CFL number of $0.48$.
\begin{figure}[H]
\centering
\includegraphics[width=16cm]{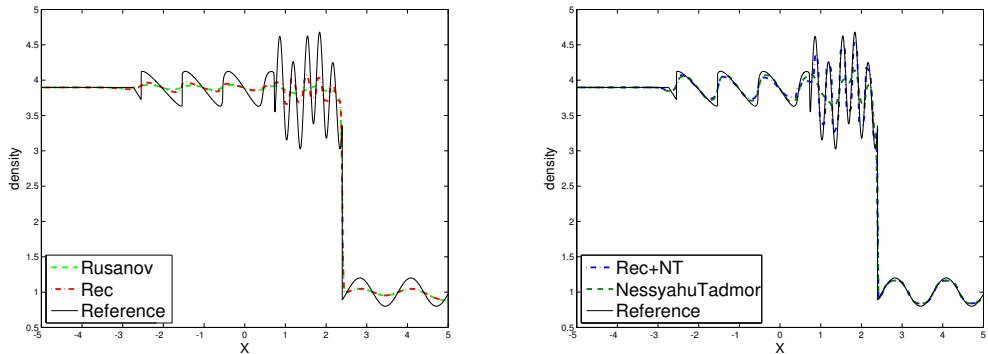}
\caption{Density at time $1.8$ for the shock entropy wave interaction}  \label{FSine}
\end{figure}
The high frequency oscillations are sharply captured.
\end{example} 

\begin{example} 
On Figure~\ref{FSMSGD}, the test case consists in a Riemann problem containing a slowly moving $3$-shock. We plot the momentum for a mesh of $800$ cells and a CFL number of $0.3$.
More precisely, we have
$$ \rho_{L}=3.86  \ u_{L}= -0.81 \ p_{L}= 10.33 \ \ \text{ and } \rho_{R}= 1.05 \ u_{R}= -3.44 \ p_{R}=1.05  \, .$$
This test case is a slight perturbation of the pure slowly moving shock introduced by~\cite{Q94}, where $\rho_{R}=p_{R}=1$.
Once again, no spurious oscillation appears. 
\begin{figure}[H]
\centering
\includegraphics[width=16cm]{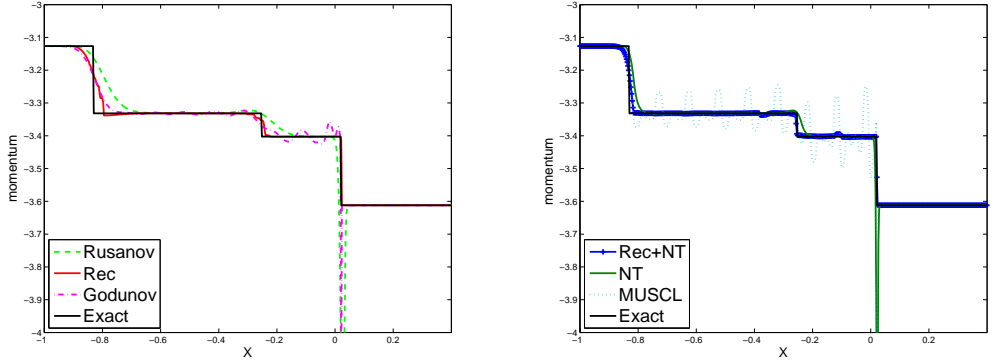}
\caption{Momentum at time $0.3$ for a Riemann problem with a slowly moving shock.}  \label{FSMSGD}
\end{figure}
\end{example} 

\begin{example} 
Another nice feature of the reconstruction scheme is that is seems to drastically diminish the wall heating phenomenon. It occurs when a shock reflects on a solid wall, and takes the form of a hollow in a density, or a spike in the temperature, near the wall. We tested our scheme on two cases considered by Donat and Marquinat in~\cite{DM96}, who proposed a cure by interlacing two schemes. For those two tests only, $\gamma=5/3$. The first case is a Riemann problem developing two symmetric shocks. The initial datum is
$$ \rho_{L}=1,  \ u_{L}= 4, \ p_{L}= 1 \ \ \text{ and } \rho_{R}= 1, \ u_{R}= -4, \ p_{R}=1 \, .$$
 The simulation is running with $200$ cells and a CFL number of $0.4$. The results, shown on Figure~\ref{FDM1}, show that the wall heating phenomenon is drastically diminish with the reconstruction scheme.
 The second test is the reflection of a gas of density $1$, pressure $0.001$ and velocity $1$ on a solid wall on its right. On Figure~\ref{FDM2} is a zoom around the wall, and we can clearly see the wall heating phenomenon and the resulting spurious oscillations for the Godunov's and MUSCL's schemes, and the good behavior of the reconstruction scheme. We took a CFL number of $0.45$ and $1 \, 000$ cells.
\begin{figure}[H]
\centering
\includegraphics[width=16cm]{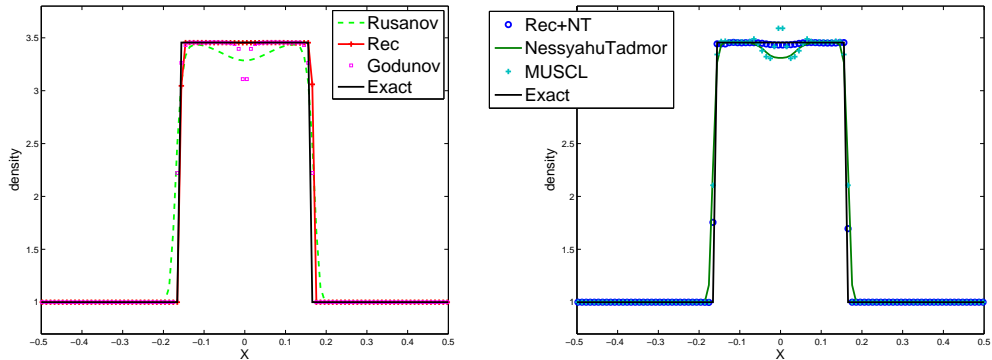}
\caption{Density and internal energy at time $0.1$ for a Riemann problem with two symmetric shocks.}  \label{FDM1}
\end{figure}
\begin{figure}[H]
\centering
\includegraphics[width=16cm]{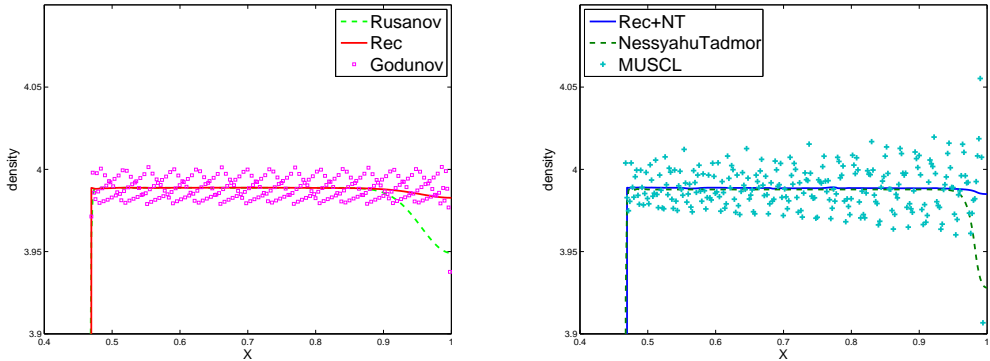}
\caption{Density and internal energy at time $1.6$ of a fluid reflecting on a solid wall.}  \label{FDM2}
\end{figure}
\end{example} 
\section{Conclusions and perspectives}
In this paper, we presented a finite volume scheme in which the mean value inside each cell is replaced, whenever it is possible, by a single shock. The fluxes are computed by letting these discontinuities evolved during the time step. For convex scalar conservation laws and for the barotropic Euler equation, we proved that the scheme is exact on pure shocks. In the latest case, we use a robust criterion to detect shocks and decide which wave should be reconstruct. The lack of such a criterion for the full gas dynamics explains why we added extra conditions in that case. However, numerical results are encouraging, especially on the problematic test cases of slowly moving shocks and shock reflections. Indeed, the spurious oscillations in the momentum and the hollow in the density are eliminated. Our main perspectives are first to find a more suitable criterion on the full gas dynamics, and second to use an approximate Riemann solver.

\phantomsection
\addcontentsline{toc}{section}{References} 
\nocite{*}
\bibliographystyle{alpha}
\bibliography{biblioReconstruction}

\end{document}

%% file: preambuleenglish.tex
\usepackage[english]{babel}
\usepackage{verbatim}
\usepackage{fancyvrb}
\usepackage[utf8x]{inputenc}
\usepackage{multicol,lipsum,float}
\usepackage[T1]{fontenc}
\usepackage{graphicx}
\usepackage{multicol}
\usepackage{graphics}
\usepackage{xcolor}
\usepackage{pgf}
\usepackage{amssymb}
\usepackage{amsmath}
\usepackage{amsfonts}
\usepackage{amssymb}
\usepackage{graphicx}
\usepackage{bm}
\usepackage{amsthm}
\usepackage{enumerate}
\usepackage{textcomp}
\usepackage{mathrsfs}
\usepackage{caption}
\usepackage{stmaryrd}
\usepackage{url}
\usepackage{psfrag}
\usepackage{hyperref}

\newtheorem{thm}{Theorem}[section]
\newtheorem{lemma}[thm]{Lemma}
\newtheorem{prop}[thm]{Proposition}

\theoremstyle{remark}
\newtheorem{rem}[thm]{Remark}
\newtheorem{example}{Example}[section]

\theoremstyle{definition}
\newtheorem{defi}[thm]{Definition}

\newcommand{\R}{\mathbb{R}}

\newcommand{\Z}{\mathbb{Z}}
\newcommand{\eps}{\epsilon}

%% file: Reconstruction.bbl
\begin{thebibliography}{ADVLCL08}

\bibitem[ADVLCL08]{AVCL08}
Fran{\c{c}}ois Alouges, Florian De~Vuyst, G{\'e}rard Le~Coq, and Emmanuel
  Lorin.
\newblock The reservoir technique: a way to make {G}odunov-type schemes zero or
  very low diffuse. {A}pplication to {C}olella-{G}laz solver.
\newblock {\em Eur. J. Mech. B Fluids}, 27(6):643--664, 2008.

\bibitem[AR97]{AR97}
Mohit Arora and Philip~L. Roe.
\newblock On postshock oscillations due to shock capturing schemes in unsteady
  flows.
\newblock {\em J. Comput. Phys.}, 130(1):25--40, 1997.

\bibitem[BCLL08]{BCLL08}
Benjamin Boutin, Christophe Chalons, Fr{\'e}d{\'e}ric Lagouti{\`e}re, and
  Philippe~G. LeFloch.
\newblock Convergent and conservative schemes for nonclassical solutions based
  on kinetic relations. {I}.
\newblock {\em Interfaces Free Bound.}, 10(3):399--421, 2008.

\bibitem[DL01]{DL02}
Bruno Despr{\'e}s and Fr{\'e}d{\'e}ric Lagouti{\`e}re.
\newblock Contact discontinuity capturing schemes for linear advection and
  compressible gas dynamics.
\newblock {\em J. Sci. Comput.}, 16(4):479--524 (2002), 2001.

\bibitem[DM96]{DM96}
Rosa Donat and Antonio Marquina.
\newblock Capturing shock reflections: an improved flux formula.
\newblock {\em J. Comput. Phys.}, 125(1):42--58, 1996.

\bibitem[Eva10]{Evans}
Lawrence~C. Evans.
\newblock {\em Partial differential equations}, volume~19 of {\em Graduate
  Studies in Mathematics}.
\newblock American Mathematical Society, Providence, RI, second edition, 2010.

\bibitem[GR91]{GR91}
Edwige Godlewski and Pierre-Arnaud Raviart.
\newblock {\em Hyperbolic systems of conservation laws}, volume 3/4 of {\em
  Math\'ematiques \& Applications (Paris) [Mathematics and Applications]}.
\newblock Ellipses, Paris, 1991.

\bibitem[JL96]{JL96}
Shi Jin and Jian-Guo Liu.
\newblock The effects of numerical viscosities. {I}. {S}lowly moving shocks.
\newblock {\em J. Comput. Phys.}, 126(2):373--389, 1996.

\bibitem[Kru70]{K70}
S~Krushkov.
\newblock First-order quasilinear equations in several independent variables.
\newblock {\em Math. USSR Sb.}, 10:217--243, 1970.

\bibitem[Lag06]{L12}
Fr{\'e}d{\'e}ric Lagouti{\`e}re.
\newblock Non-dissipative entropy satisfying discontinuous reconstruction
  schemes for hyperbolic conservation laws.
\newblock {\em Preprint}, 16, 2006.

\bibitem[Lag08]{L08}
Fr{\'e}d{\'e}ric Lagouti{\`e}re.
\newblock Stability of reconstruction schemes for scalar hyperbolic
  conservation laws.
\newblock {\em Commun. Math. Sci.}, 6(1):57--70, 2008.

\bibitem[LW03]{LW03}
Richard Liska and Burton Wendroff.
\newblock Comparison of several difference schemes on 1{D} and 2{D} test
  problems for the {E}uler equations.
\newblock {\em SIAM J. Sci. Comput.}, 25(3):995--1017 (electronic), 2003.

\bibitem[LWM08]{LWM08}
Hongxia Li, Zhigang Wang, and De-kang Mao.
\newblock Numerically neither dissipative nor compressive scheme for linear
  advection equation and its application to the {E}uler system.
\newblock {\em J. Sci. Comput.}, 36(3):285--331, 2008.

\bibitem[Men94]{M94}
Ralph Menikoff.
\newblock Errors when shock waves interact due to numerical shock width.
\newblock {\em SIAM J. Sci. Comput.}, 15(5):1227--1242, 1994.

\bibitem[Noh87]{N87}
W.F Noh.
\newblock Errors for calculations of strong shocks using an artificial
  viscosity and an artificial heat flux.
\newblock {\em J. Comp. Phys.,}, 72(1):78--120, 1987.

\bibitem[NT90]{NT90}
Haim Nessyahu and Eitan Tadmor.
\newblock Nonoscillatory central differencing for hyperbolic conservation laws.
\newblock {\em J. Comput. Phys.}, 87(2):408--463, 1990.

\bibitem[Qui94]{Q94}
James~J. Quirk.
\newblock A contribution to the great {R}iemann solver debate.
\newblock {\em Internat. J. Numer. Methods Fluids}, 18(6):555--574, 1994.

\bibitem[Rob90]{R90}
Thomas~W. Roberts.
\newblock The behavior of flux difference splitting schemes near slowly moving
  shock waves.
\newblock {\em J. Comput. Phys.}, 90(1):141--160, 1990.

\bibitem[Tor09]{Toro}
Eleuterio~F. Toro.
\newblock {\em Riemann solvers and numerical methods for fluid dynamics}.
\newblock Springer-Verlag, Berlin, third edition, 2009.
\newblock A practical introduction.

\bibitem[vL97]{VL79}
Bram van Leer.
\newblock Towards the ultimate conservative difference scheme. {V}. {A}
  second-order sequel to {G}odunov's method.
\newblock {\em J. Comput. Phys.}, 135(2):227--248, 1997.

\bibitem[WC84]{CW84}
Paul Woodward and Phillip Colella.
\newblock The numerical simulation of two-dimensional fluid flow with strong
  shocks.
\newblock {\em J. Comput. Phys.}, 54(1):115--173, 1984.

\end{thebibliography}
